\documentclass[11pt,a4paper,twoside,reqno]{amsart}

\usepackage[T1]{fontenc}
\usepackage{lmodern}
\usepackage[
a4paper,
inner=31mm,
outer=34mm,
top=30mm,
bottom=32mm,
headsep=8mm
]{geometry}
\usepackage{amsmath,amssymb,mathtools,bm}
\usepackage[expansion=false]{microtype}
\usepackage{graphicx}
\usepackage{placeins}
\usepackage{comment}
\usepackage[dvipsnames]{xcolor}
\usepackage{hyperref}
\usepackage[nameinlink,capitalise,noabbrev]{cleveref}
\usepackage[english]{babel}
\hypersetup{
colorlinks=true,
linkcolor=red,
citecolor=green,
pdfborder={0 0 0},
bookmarksopen=true,
bookmarksnumbered=true,
pdftitle={Equal-Area Cubed-Sphere Partitions and Spherical Designs},
pdfauthor={Anonymous},
pdfsubject={Spherical designs from equal-area cubed-sphere partitions},
pdfkeywords={spherical design, cubed sphere, Karcher center, Bernstein inequality, Brouwer degree}
}

\theoremstyle{plain}
\newtheorem{theorem}{Theorem}[section]
\newtheorem{lemma}[theorem]{Lemma}

\theoremstyle{definition}

\newtheorem{celldefinition}{Definition}[section]

\newtheorem{remark}[theorem]{Remark}
\crefname{celldefinition}{Definition}{Definitions}

\numberwithin{equation}{section}

\newcommand{\Sph}{\mathbb S^2}
\newcommand{\R}{\mathbb R}
\newcommand{\dd}{\,\mathrm d}
\newcommand{\dist}{\operatorname{dist}}
\newcommand{\grad}{\nabla}
\newcommand{\Rot}{\mathcal L}
\newcommand{\op}{\mathrm{op}}
\newcommand{\tr}{\operatorname{tr}}
\title[Spherical $t$-Designs on $\mathbb S^2$ with $54t^2$ Points]{Spherical $t$-Designs on $\mathbb S^2$ with $54t^2$ Points}
\author{Zhiqiang Xu}
\address{State Key Laboratory of Mathematical Sciences, Academy of Mathematics and Systems Science, Chinese Academy of Sciences, Beijing 100190, China;    School of Mathematical Sciences, University of Chinese Academy of Sciences, Beijing 100049, China. }
\email{xuzq@lsec.cc.ac.cn}
\author{Zili Xu}
\address{School of Mathematical Sciences,  Key Laboratory of MEA (Ministry of Education), Shanghai Key Laboratory of PMMP, and Nantong Institute for Applied Mathematics and Artificial Intelligence, East China Normal University, Shanghai 200241, China}
\email{zlxu@math.ecnu.edu.cn}

\date{}

\begin{document}

\begin{abstract}
We prove that, for every integer $t\geq 1$, the unit two-sphere admits a spherical $t$-design consisting of exactly $54t^2$ points. More generally, such a design exists with exactly $6q^2$ points for every integer $q\geq 3t$. The proof builds on the topological degree method of Bondarenko, Radchenko and Viazovska, using an explicit equal-area partition of the sphere based on the map of Ro\c{s}ca and Plonka. By choosing the cell centers to minimize the average squared geodesic distance and deriving sharper sampling estimates, we obtain the stated quadratic bound on the size of spherical $t$-designs on $\Sph$.

\end{abstract}
\maketitle

\section{Introduction}
\label{sec:introduction}

Let $\mathbb S^d=\{\bm x\in\mathbb R^{d+1}:\|\bm x\|_{\ell_2}=1\}$ be equipped with normalized surface measure $\sigma_d$.  A set of points $\bm x_1,\ldots,\bm x_N\in \mathbb S^d$ is a \emph{spherical $t$-design} if
\[
\frac1N\sum_{i=1}^N P(\bm x_i)
=\int_{\mathbb S^d}P(\bm x)\dd\sigma_d(\bm x)
\]
for every polynomial $P(\bm x)$ in $d+1$ variables of total degree at most $t$. Thus spherical designs are equal-weight cubature formulas that exactly integrate every spherical polynomial up to a prescribed degree.  The notion of spherical designs was introduced by Delsarte, Goethals, and Seidel~\cite{DGS1977}.

Write $\mathcal N(d,t)$ for the smallest possible number of points in a spherical $t$-design on $\mathbb S^d$.  The Fisher-type lower bound of Delsarte, Goethals, and Seidel~\cite{DGS1977} gives
\[
\mathcal N(d,t)\ge
\begin{cases}
\displaystyle\ \binom{d+k}{d}+\binom{d+k-1}{d},&\text{if $t=2k$},\\[5pt]
\displaystyle\ 2\ \binom{d+k}{d},&\text{if $t=2k+1$}.
\end{cases}
\]
In particular, when $d=2$ we have
\[
\mathcal N(2,t)\ge
\frac{1}{4}t^2+o(t^2),
\qquad\text{as }t\to\infty.
\]

Seymour and Zaslavsky~\cite{SeymourZaslavsky1984} proved that, for each $d,t\geq 1$, spherical $t$-designs on $\mathbb S^d$ exist with every sufficiently large number of points. Following earlier quantitative constructions by Wagner and Volkmann~\cite{Wagner1991} and Bajnok~\cite{Bajnok1992}, Korevaar and Meyers~\cite{KorevaarMeyers1993} established
\[
\mathcal N(d,t)\le C_d t^{d(d+1)/2},
\]
where $C_d>0$ depends only on $d$. They further conjectured that the exponent could be reduced to $d$, which would give the optimal order of growth in $t$. Bondarenko and Viazovska~\cite{BondarenkoViazovska2010} developed a construction of spherical designs using partitions of the sphere, sampling estimates, and the Brouwer fixed-point theorem. Bondarenko, Radchenko, and Viazovska then proved the optimal-order existence theorem~\cite{BRV2013}: for each fixed $d$ there is a constant $C_d>0$ such that, whenever $N\ge C_d t^d$, there exists a spherical $t$-design with exactly $N$ points.  They also proved that the points may be chosen well separated~\cite{BRV2015}. The same architecture has subsequently been extended to compact algebraic manifolds~\cite{EtayoMarzoOrtega2018}, compact Riemannian manifolds~\cite{GariboldiGigante2021}, and cubature formulas with prescribed positive weights~\cite{EhlerEtayoGariboldiGigantePeter2021}. 

\subsection{Spherical $t$-designs on $\mathbb S^2$}
We focus on the two-dimensional sphere $\mathbb S^2\subset\mathbb R^3$. As an idealized model of the Earth's surface and a natural representation of directions in three-dimensional space, $\mathbb S^2$ is an important domain for numerical integration in geophysics and other applications involving directional data. Its connections with polyhedral geometry also motivate the study of spherical designs with few points. Our aim is to obtain a small explicit constant $C_2$ such that $\mathcal N(2,t)\leq C_2t^2$ for every integer $t\geq 1$. We begin by reviewing earlier work on spherical $t$-designs on $\mathbb S^2$, with particular emphasis on bounds for the number of points.

Hardin and Sloane \cite{HardinSloane1996} carried out an extensive study of spherical designs on $\mathbb S^2$, producing several exact constructions as well as highly accurate numerical examples.  Their computations led them to a conjecturally infinite family of spherical $t$-designs whose cardinalities satisfy
\[
N=\frac{1}{2}t^2+o(t^2)
\qquad\text{as }t\to\infty.
\]
Consequently,  their evidence suggests the conjectural asymptotic upper bound
\[
{\mathcal N}(2,t)\leq \frac{1}{2}t^2+o(t^2).
\]

Another conjecture, formulated by Chen and Womersley \cite{ChenWomersley2006}, asserts that for every positive integer $t$, there exists a spherical $t$-design on $\mathbb S^2$ consisting of exactly $N=(t+1)^2$ points. More precisely, they conjectured the existence of such a design in a neighborhood of an extremal fundamental system.  Using rigorous interval-arithmetic computations, Chen, Frommer, and Lang \cite{ChenFrommerLang2011} subsequently verified this conjecture for $1\leq t\leq 100$, whereas the assertion for arbitrary $t$ remains open.

Later,  Zhou and Chen introduced spherical $t_\varepsilon$-designs, which permit positive weights close to equal weights, and proved approximation results for this more flexible notion~\cite{ZhouChen2018}.  Womersley reported extensive numerical equal-weight configurations with favorable separation and covering properties, including general designs through strength $180$ and antipodal designs through strength $325$~\cite{Womersley2018}.  

These works provide valuable constructions and numerical evidence supporting the existence of spherical designs with few points, but the smallest constant $C_2$ for which $\mathcal N(2,t)\leq C_2t^2$ holds for every integer $t\geq 1$ remains unknown.

\subsection{Constants in the Bondarenko-Radchenko-Viazovska bound
on $\mathbb S^2$}
\label{sec:brv-constants}

Bondarenko, Radchenko, and Viazovska proved that for each fixed dimension $d$, there exists a constant $C_d>0$ such that a spherical $t$-design on $\mathbb S^d$ with exactly $N$ points exists whenever $N\geq C_dt^d$~\cite{BRV2013}. Their proof gives the sufficient condition
\[
C_d>\left(\frac{54dK_d}{r_d}\right)^d.
\]
Here, $K_d$ is the smallest constant such that for every $N$, there exists an equal-area partition $(R_1,\ldots,R_N)$ of $\mathbb S^d$ whose cells have geodesic diameters at most $K_dN^{-1/d}$. The constant $r_d>0$ is a sampling threshold obtained by applying Theorem~3.1 of Mhaskar, Narcowich, and Ward~\cite{MNW2001} with $\eta=1/2$ and $p=1$. The restriction in \cite[equation~(3.26)]{MNW2001} gives $r_d\leq 1/2$. Consequently, any constant $C_d$ satisfying the above sufficient condition must also satisfy $C_d>(108dK_d)^d$. In particular, $C_2>(216K_2)^2$.

Rakhmanov, Saff and Zhou~\cite{RakhmanovSaffZhou1994} proved that for every $N\geq 2$, there exists an equal-area partition $(R_1,\ldots,R_N)$ of $\mathbb S^2$ whose cells have Euclidean diameters at most $7N^{-1/2}$.  They also showed that any uniform Euclidean diameter constant is at least $4$. Since Euclidean distance does not exceed geodesic distance, we have $K_2\ge4$. Consequently, any constant $C_2$ satisfying the Bondarenko-Radchenko-Viazovska sufficient condition necessarily obeys
\[
C_2>
 (216\cdot4)^2
=746496.
\]

\subsection{Our contributions}

We combine an explicit cubed-sphere partition with sharper sampling estimates to establish the bound
\[
\mathcal N(2,t)\leq 54t^2
\qquad\forall t\geq 1.
\]
More precisely, we prove the following theorem.

\begin{theorem}\label{thm:coarse-main}
For every integer $t\geq 1$ and every integer $q\geq 3t$, there exists a spherical $t$-design on $\Sph$ with exactly $6q^2$ points. In particular, taking $q=3t$ gives a spherical $t$-design with exactly $54t^2$ points.
\end{theorem}

The constant $54$ is not optimal within this framework.  We choose the constant $54$ to keep the proof transparent and its quantitative estimates verifiable. All estimates follow from analytic inequalities without interval-arithmetic certificates.  This choice allows us to present the main ideas of the method clearly, without lengthy optimizations of intermediate constants.  By refining the geometric and distance-moment estimates for the partition cells and incorporating rigorous computer-assisted calculations, we can reduce the constant to $35$. The smallest constant attainable within this framework remains unknown. A natural question is whether further refinements can establish $\mathcal N(2,t)\le C_2t^2$ for all sufficiently large integers $t$, with an explicit constant $C_2\le1$.

Bondarenko, Radchenko, and Viazovska proved that there exists a constant $C_2>0$ such that, for every integer $t\geq 1$ and every integer $N\geq C_2t^2$, a spherical $t$-design on $\Sph$ with exactly $N$ points exists \cite{BRV2013}. Their result covers every integer above the stated threshold. Theorem~\ref{thm:coarse-main} applies to the node counts $N=6q^2$ with $q\geq 3t$ and gives the explicit bound $\mathcal N(2,t)\leq 54t^2$. The coefficient $54$ is substantially smaller than the threshold required by the particular estimates in Section~\ref{sec:brv-constants}.

It is possible to extend our argument to every integer $N\geq 54t^2$ by refining the partition construction. Such an extension would require equal-area partitions into $N$ cells with geometric and moment estimates strong enough to satisfy the positivity criterion in our proof. We leave this question for future work.

\subsection{Proof strategy and refinements of the Bondarenko-Radchenko-Viazovska approach}
\label{sec:degree-criterion}

Our proof follows the topological framework of Bondarenko, Radchenko, and Viazovska ~\cite[Section~2]{BRV2013} (see also \cite[Section~2]{BRV2015}). The idea is to associate with each mean-zero spherical polynomial a configuration of points that depends continuously on the polynomial. If the average value of the polynomial at these points is positive on the boundary of a suitable norm ball, a topological argument produces a configuration satisfying all the spherical design conditions simultaneously. The relevant topological result is the following consequence of Brouwer degree theory~\cite{ORegan2006}.

\begin{theorem}\label{thm:brouwer-degree-criterion}
Let $\Omega\subset\R^n$ be a bounded open set containing the origin, and let $f:\overline{\Omega}\to\R^n$ be continuous. If
\[
\langle \bm{x},f(\bm{x})\rangle>0
\qquad\forall\bm{x}\in\partial\Omega,
\]
then there exists $\bm{x}_0\in\Omega$ such that $f(\bm{x}_0)=0$.
\end{theorem}

To apply this theorem, let $\Pi_t^0$ denote the real vector space of spherical polynomials of degree at most $t$ with zero mean with respect to normalized surface measure $\sigma=\sigma_2$ on $\Sph$. Equip $\Pi_t^0$ with the inner product
\[
\langle P,Q\rangle_*
:=\int_{\Sph}P(\bm{x})Q(\bm{x})\,\dd\sigma(\bm{x}).
\]
For each $\bm{x}\in\Sph$, the Riesz representation theorem gives a unique polynomial $G_{\bm{x}}\in\Pi_t^0$ satisfying
\[
\langle G_{\bm{x}},Q\rangle_*=Q(\bm{x})
\qquad\forall Q\in\Pi_t^0.
\]
Thus, for distinct points $\bm{x}_1,\ldots,\bm{x}_N\in\Sph$, the spherical $t$-design conditions are equivalent to the single vector equation
\[
\sum_{i=1}^N G_{\bm{x}_i}=0.
\]
Indeed, this equation is equivalent to $\sum_{i=1}^N Q(\bm{x}_i)=0$ for every $Q\in\Pi_t^0$, while the quadrature identity for constant polynomials holds automatically.

Define
\[
\|P\|
:=\int_{\Sph}\|\grad P(\bm{x})\|_{\ell_2} \,\dd\sigma(\bm{x})
\quad\text{and}\quad
\Omega:=\{P\in\Pi_t^0:\|P\|<1\}.
\]
The quantity $\|\cdot\|$ is a norm on $\Pi_t^0$, because a polynomial with vanishing spherical gradient is constant, and its zero mean then forces it to vanish. Since $\Pi_t^0$ is finite-dimensional, $\Omega$ is a bounded open neighborhood of the origin. It therefore suffices to construct continuous maps
\begin{equation}\label{xueq80}
\bm{y}_i:\overline{\Omega}\to\Sph,
\qquad i=1,\ldots,N,
\end{equation}
such that the points $\bm{y}_1(P),\ldots,\bm{y}_N(P)$ are pairwise distinct for each $P\in\overline{\Omega}$ and
\begin{equation}\label{xueq90}
\sum_{i=1}^N P\bigl(\bm{y}_i(P)\bigr)>0
\qquad\forall P\in\partial\Omega.
\end{equation}
To see this, define
\[
f(P):=\sum_{i=1}^N G_{\bm{y}_i(P)}.
\]
The map $\bm{x}\mapsto G_{\bm{x}}$ is continuous, as is seen by expressing $G_{\bm{x}}$ in an orthonormal basis of $\Pi_t^0$. Hence $f$ is continuous, and
\[
\langle P,f(P)\rangle_*
=\sum_{i=1}^N P\bigl(\bm{y}_i(P)\bigr)>0
\qquad\forall P\in \partial\Omega.
\]
After identifying $\Pi_t^0$ with a Euclidean space through an orthonormal basis, Theorem~\ref{thm:brouwer-degree-criterion} gives $P_0\in\Omega$ with $f(P_0)=0$. The corresponding points $\bm{y}_1(P_0),\ldots,\bm{y}_N(P_0)$ then form a spherical $t$-design.

The main analytic task is therefore to establish \eqref{xueq90}. In the framework of Bondarenko, Radchenko, and Viazovska, one starts with an area-regular partition, chooses an arbitrary initial point in each cell, and moves all points along the same regularized gradient vector field associated with $P$. Sampling estimates bound the initial quadrature error and give a lower bound for the increase in the average value of $P$ along the flow. Positivity follows once this increase exceeds the initial quadrature error.

We follow this framework but introduce several refinements that yield a substantially smaller explicit value of \(N\). The main differences are summarized as follows.

\begin{enumerate}
\item[{\rm (i)}] \emph{Explicit geometry of the partition.} Bondarenko, Radchenko, and Viazovska  used any area-regular partition satisfying a dimension-dependent diameter bound. In this paper, we use the equal-area map of Ro\c{s}ca and Plonka to construct a partition of $\Sph$ into $6q^2$ cubed-sphere cells. Its explicit formulas allow us to bound the distances from the cell centers and the second distance moments (see Lemma~\ref{lem:karcher-localization}). These bounds also control higher distance moments and provide explicit geometric input for the sampling estimates.

\item[{\rm (ii)}] \emph{Cancellation through the choice of initial points.} Bondarenko, Radchenko, and Viazovska used an arbitrary initial point in each cell. Instead, we choose the Karcher center, namely the point minimizing the average squared geodesic distance to the cell. The first-variation identity at this center makes the average tangent displacement vanish. Consequently, the linear term in the cellwise Taylor expansion integrates to zero. The initial quadrature error is therefore controlled by second and higher distance moments, improving on the first-order displacement estimate used in  Bondarenko, Radchenko, and Viazovska .

\item[{\rm (iii)}] \emph{Sampling estimates based on cell moments.} Bondarenko, Radchenko, and Viazovska  obtained their sampling bounds from a general Marcinkiewicz-Zygmund inequality controlled by the maximum cell diameter. We derive the required estimates directly from rotational Bernstein inequalities and the distance moments of the cells. We then use explicit Hessian bounds and a differential inequality to control the sampling quantities as the points move. These estimates yield an explicit criterion ensuring that the gain along the flow exceeds the initial quadrature error.
\end{enumerate}

Thus, the topological mechanism is inherited from  Bondarenko, Radchenko, and Viazovska, while Karcher centering and the moment-based sampling analysis provide additional quantitative information beyond an evaluation of the constants in their proof.

\subsection{Organization of the paper}

Section~\ref{sec:preliminaries} collects the geometric and polynomial identities used in the proof.  Section~\ref{sec:map} constructs the cubed-sphere cells and states their Karcher localization estimates. Section~\ref{sec:sample} establishes the sampling bounds, and Section~\ref{sec:tag-criterion} combines them with the regularized flow to prove Theorem~\ref{thm:coarse-main}.  The appendices give the proofs of the rotational estimates and the cell localization bounds.

\section{Notation and analytic estimates}\label{sec:preliminaries}

This section introduces the notation and analytic estimates used to bound the initial quadrature error and to control the sampled polynomial values as the points move along the gradient flow. After fixing the geometric notation and polynomial spaces, we establish a geodesic Taylor formula for comparing cell averages with values at the cell centers, together with a Hessian bound for the variation of gradient length along curves. We then collect rotational derivative identities and Bernstein inequalities that allow these estimates to be applied repeatedly without increasing the polynomial degree. These tools form the analytic basis for the sampling estimates and the positivity criterion developed later.

\subsection{\texorpdfstring{The sphere $\Sph$}{The sphere S2}}

Throughout the paper,
\[
\Sph=\{\bm x\in\R^3:\|\bm x\|_{\ell_2}=1\}
\]
is equipped with the round metric. We write $\sigma=\sigma_2$ for the normalized surface measure, i.e.,
\[
\sigma(\Sph)=\frac1{4\pi}\int_{\Sph}\dd S=1,
\]
where $\dd S$ denotes ordinary surface area measure.  We write $\langle\cdot,\cdot\rangle$ for the Euclidean inner product.  For two points $\bm{x},\bm{y}\in\Sph$, we denote their geodesic distance by 
\begin{equation*}
\dist(\bm{x},\bm{y})=\arccos \langle \bm x,\bm y\rangle\in[0,\pi].	
\end{equation*}
The open geodesic ball with center $\bm x$ and radius $r$ is denoted by
\[
B(\bm x,r)=\{\bm y\in\Sph:\dist(\bm x,\bm y)<r\}.
\]
The tangent plane at $\bm x\in\Sph$ is
\[
T_{\bm x}\Sph
=\{\bm v\in\R^3:\langle\bm v,\bm x\rangle=0\}.
\]
For $\bm{c},\bm{y}\in\Sph$ with $r=\dist(\bm{c},\bm{y})<\pi$, let 
\begin{equation}
\xi_{\bm{c}}(\bm{y})
:=\log_{\bm c}(\bm y)=
\begin{cases}
\displaystyle
\frac{r}{\sin r}\cdot (\bm I_3-\bm c\bm c^T)\bm y
,&\bm{y}\ne \bm{c},\\[1ex]
\bm{0},&\bm{y}=\bm{c}.
\end{cases}
\label{eq:tangent-displacement}
\end{equation}
denote the Riemannian logarithm of $\bm y$ at $\bm c$ \cite{Afsari2011,Karcher1977}. Note that $\|\xi_{\bm{c}}(\bm{y})\|_{\ell_2}=r$, and $\xi_{\bm{c}}(\bm{y})\in T_{\bm{c}}\Sph$  if $\bm y\neq \bm c$.  We denote the shortest geodesic from $\bm c$ to $\bm y$ by
\begin{equation}
\gamma_{\bm{c},\bm{y}}(s)=
\begin{cases}
\displaystyle
\frac{\sin((1-s)r)}{\sin r}\,\bm{c}
+\frac{\sin(sr)}{\sin r}\,\bm{y},&\bm{y}\ne \bm{c},\\[1ex]
\bm{c},&\bm{y}=\bm{c},
\end{cases}
\qquad 0\le s\le1.
\label{eq:radial-geodesic}
\end{equation}
Note that $\gamma_{\bm{c},\bm{y}}(0)=\bm c$ and $\gamma_{\bm{c},\bm{y}}(1)=\bm y$. It has constant speed $r$ and $\frac{\dd}{\dd s} \gamma_{\bm c,\bm y}(s)|_{s=0} =\xi_{\bm c}(\bm y)$.

\subsection{The polynomial spaces}

The Euclidean gradient and Hessian matrix of a function $f(\bm{x})$ are
\[
\nabla_{\R^3}f
=\begin{pmatrix}
\partial_{x_1}f\\
\partial_{x_2}f\\
\partial_{x_3}f
\end{pmatrix},
\qquad
\nabla^2_{\R^3}f
:=\bigl(\partial_{x_i}\partial_{x_j}f\bigr)_{i,j=1}^3.
\]
For ${\bm x}\in \Sph$, the spherical gradient $\grad f(\bm x)$ is 
\[
\grad f(\bm x)
=(\bm I_3-\bm x\bm x^T)\nabla_{\R^3}f(\bm x)
\in T_{\bm x}\Sph.
\]
The spherical Hessian matrix $\nabla^2 f(\bm x)$, viewed as a self-adjoint operator on $T_{\bm x}\Sph$ and represented on $\R^3$ by extending it as zero on the normal line, is
\[
\nabla^2 f(\bm x)
:=(\bm I_3-\bm x\bm x^T)\nabla^2_{\R^3}f(\bm x)(\bm I_3-\bm x\bm x^T)
-\langle \nabla_{\R^3}f(\bm x),\bm x\rangle
(\bm I_3-\bm x\bm x^T).
\]
If $\Phi(\bm x)=(\Phi_1(\bm x),\ldots,\Phi_n(\bm x)): \R^m\to\R^n$ is differentiable, its Jacobian matrix at $\bm p\in \R^m$ is
\[
\bm J_{\Phi}(\bm p)
=\left(\frac{\partial\Phi_i}{\partial x_j}(\bm p)\right)_{
\substack{1\le i\le n\\1\le j\le m}}
\in\R^{n\times m}.
\]
For a matrix $\bm M$, its operator norm is
\[
\|\bm M\|_{\op}
:=\sup_{\bm v\ne\bm0}
\frac{\|\bm M\bm v\|_{\ell_2}}{\|\bm v\|_{\ell_2}}.
\]

Recall that, for an integer $t\geq 1$, $\Pi_t$ denotes the space of restrictions to $\Sph$ of real polynomials on $\R^3$ of total degree at most $t$, and
\[
\Pi_t^0=\left\{P\in\Pi_t:
\int_{\Sph}P(\bm{x})\,\dd\sigma(\bm{x})=0\right\}.
\]
We also recall the notation
\[
\|P\|=\int_{\Sph}\|\grad P(\bm{x})\|_{\ell_2}\,\dd\sigma(\bm{x}),
\qquad P\in\Pi_t,
\]
which defines a seminorm on $\Pi_t$ and a norm on $\Pi_t^0$.

The following lemma is the spherical analogue of the usual Taylor formula with a second-order integral remainder, with the line segment replaced by the shortest geodesic joining the two points. It is the $k=1$, $t=1$ geodesic specialization of the Riemannian Taylor formula in Le Brigant and Puechmorel~\cite[Lemma~1, p.~20]{LeBrigantPuechmorel2019}.
\begin{lemma}
\label{lem:geodesic-taylor-second-order}
Let $Q\in\Pi_t$, and let $\bm c,\bm y\in\Sph$ satisfy $\dist(\bm c,\bm y)<\pi$.  Then
\[
\begin{aligned}
Q(\bm y)-Q(\bm c)
&=\langle\grad Q(\bm c),\xi_{\bm c}(\bm y)\rangle
+\int_0^1(1-\tau)
\left\langle
\nabla^2 Q(\gamma(\tau))
\gamma'(\tau),
\gamma'(\tau)
\right\rangle\dd\tau,
\end{aligned}
\]
where $\gamma(\tau):=\gamma_{\bm c,\bm y}(\tau)$ and $\gamma'(\tau)=\frac{\dd}{\dd\tau}\gamma(\tau)$.
\end{lemma}

\begin{proof}
If $\bm y=\bm c$, then $\xi_{\bm c}(\bm y)=\bm0$ and the geodesic is constant, so the identity is immediate.  Assume $\bm y\ne\bm c$. Denote  $g(\tau)=Q(\gamma(\tau))$.  The one-dimensional Taylor formula with integral remainder gives
\[
g(1)-g(0)=g'(0)+\int_0^1(1-\tau)g''(\tau)\dd\tau.
\]
By the chain rule and the definition of the spherical gradient,
\[
g'(0)
=\left\langle\grad Q(\gamma(0)),\gamma'(0)\right\rangle
=\langle\grad Q(\bm c),\xi_{\bm c}(\bm y)\rangle.
\]
It remains only to identify $g''(\tau)$.  Since $\gamma$ is a constant speed great-circle geodesic on $\Sph$, we have
\[
\left\langle\gamma(\tau),\gamma'(\tau)\right\rangle=0
\quad\text{and}\quad
\frac{\dd^2\gamma}{\dd\tau^2}(\tau)
=-\left\|\gamma'(\tau)\right\|_{\ell_2}^2\gamma(\tau).
\]
Differentiating $g(\tau)=Q(\gamma(\tau))$ in $\R^3$ gives
\[
\begin{aligned}
g''(\tau)
&=
\left\langle
\nabla^2_{\R^3}Q(\gamma(\tau))
\gamma'(\tau),
\gamma'(\tau)
\right\rangle +
\left\langle
\nabla_{\R^3}Q(\gamma(\tau)),
\frac{\dd^2\gamma}{\dd\tau^2}(\tau)
\right\rangle\\
&=
\left\langle
\nabla^2_{\R^3}Q(\gamma(\tau))
\gamma'(\tau),
\gamma'(\tau)
\right\rangle
-\left\langle\nabla_{\R^3}Q(\gamma(\tau)),\gamma(\tau)\right\rangle
\left\|\gamma'(\tau)\right\|_{\ell_2}^2.
\end{aligned}
\]
Because $\gamma'(\tau)\in T_{\gamma(\tau)}\Sph$, the definition of the spherical Hessian in $\R^3$ gives
\[
g''(\tau)
=
\left\langle
\nabla^2 Q(\gamma(\tau))
\gamma'(\tau),
\gamma'(\tau)
\right\rangle.
\]
Substituting the formulas for $g'(0)$ and $g''(\tau)$ into the preceding one-dimensional Taylor formula proves the claim.
\end{proof}

The following lemma bounds the variation in the norm of the spherical gradient along a curve on the sphere.

\begin{lemma}
\label{lem:spherical-gradient-variation-curve}
Let $H$ be a twice continuously differentiable function on $\Sph$, and let $\gamma:[a,b]\to\Sph$ be a smooth curve. Then
\[
\bigg|
\|\grad H(\gamma(b))\|_{\ell_2}
-
\|\grad H(\gamma(a))\|_{\ell_2}
\bigg|
\le
\int_a^b
\bigg\|\nabla^2H(\gamma(r))\bigg\|_{\op}
\bigg\| \frac{\dd}{\dd r}\gamma(r)\bigg\|_{\ell_2}\,\dd r.
\]

\end{lemma}

\begin{proof}
Regard the spherical gradient along $\gamma$ as a $C^1$ curve in $\R^3$, and write
\[
\bm g(r)=\grad H(\gamma(r))
\quad\text{and}\quad
\bm v(r)=\frac{\dd\gamma}{\dd r}(r).
\]
Both $\bm g(r)$ and $\bm v(r)$ are tangent to $\Sph$ at $\gamma(r)$.  By \cite[Example~5.32, equation~(5.19)]{Boumal2023}, we have
\[
\left(\bm I_3-\gamma(r)\gamma(r)^T\right)
\frac{d\bm g}{dr}(r)
=
\nabla^2 H(\gamma(r))\bm v(r).
\]
Since $\bm g(r)\in T_{\gamma(r)}\mathbb S^2$, taking the inner product with $\bm g(r)$ gives
\begin{equation}\label{xueq60}
\left\langle\bm g(r),\frac{\dd\bm g}{\dd r}(r)\right\rangle
=\left\langle\bm g(r),\nabla^2H(\gamma(r))\bm v(r)\right\rangle.
\end{equation}
Fix $\varepsilon>0$ and set
\[
f_\varepsilon(r)
=\sqrt{\|\bm g(r)\|_{\ell_2}^2+\varepsilon^2}.
\]
This is a $C^1$ function on $[a,b]$.  The Euclidean chain rule and the preceding identity \eqref{xueq60} give
\[
\frac{\dd f_\varepsilon}{\dd r}(r)
=\frac{\left\langle\bm g(r),\frac{\dd\bm g}{\dd r}(r)\right\rangle}
{\sqrt{\|\bm g(r)\|_{\ell_2}^2+\varepsilon^2}}
=\frac{\left\langle\bm g(r),
\nabla^2H(\gamma(r))\bm v(r)\right\rangle}
{\sqrt{\|\bm g(r)\|_{\ell_2}^2+\varepsilon^2}}.
\]
By the Cauchy-Schwarz inequality and the definition of the operator norm,
\[
\begin{aligned}
\left|\frac{\dd f_\varepsilon}{\dd r}(r)\right|
&\le
\frac{\|\bm g(r)\|_{\ell_2}}
{\sqrt{\|\bm g(r)\|_{\ell_2}^2+\varepsilon^2}}
\|\nabla^2H(\gamma(r))\|_{\op}\,\|\bm v(r)\|_{\ell_2}\le\|\nabla^2H(\gamma(r))\|_{\op}\,\|\bm v(r)\|_{\ell_2}.
\end{aligned}
\]
The fundamental theorem of calculus and the triangle inequality now imply
\[
\begin{aligned}
|f_\varepsilon(b)-f_\varepsilon(a)|
&\le\int_a^b
\left|\frac{\dd f_\varepsilon}{\dd r}(r)\right|\dd r\le\int_a^b
\|\nabla^2H(\gamma(r))\|_{\op}
\left\|\frac{\dd\gamma}{\dd r}(r)\right\|_{\ell_2}\dd r.
\end{aligned}
\]
Since the right-hand side is independent of $\varepsilon$, letting $\varepsilon\downarrow0$ gives the asserted inequality, including the cases in which the gradient vanishes at an endpoint or along the curve.
\end{proof}

\subsection{Rotational derivative}

For $\bm a=(a_1,a_2,a_3)^T$ and $\bm b=(b_1,b_2,b_3)^T$ in $\R^3$, their cross product is
\begin{equation}
\bm a\times\bm b
=(a_2b_3-a_3b_2,a_3b_1-a_1b_3,a_1b_2-a_2b_1)^T.
\label{eq:cross-product-coordinate}
\end{equation}
For $\bm u\in\Sph$, let $\mathcal R_{\bm u}(s):\R^3\to\R^3$ be the right-handed rotation operator through angle $s$ about the axis $\bm u$.  Rodrigues' formula is
\[
{
\mathcal R_{\bm u}(s)\bm x
=(\cos s)\,\bm x+(\sin s)\,(\bm u\times\bm x)
+(1-\cos s)\langle\bm u,\bm x\rangle\bm u.}
\]
For a polynomial $P\in\Pi_t$  and for an axis vector $\bm u\in \Sph$, define the rotational derivative
\begin{equation}
\Rot_{\bm u}P(\bm x)
:=\left.\frac{\dd}{\dd s}
P\bigl(\mathcal R_{\bm u}(s)\bm x\bigr)\right|_{s=0},
\label{eq:rotational-derivative}
\end{equation}
which measures how fast the value of $P$ changes if we rotate $\bm x$ around the axis $\bm u$. Write $\nabla_{\R^3} P(\bm x)=(g_1,g_2,g_3)^T$, $\bm x=(x_1,x_2,x_3)$ and $\bm u=(u_1,u_2,u_3)$. Using the chain rule, we can rewrite $\Rot_{\bm u}P(\bm x)$ as
\begin{equation}
\begin{aligned}
\Rot_{\bm u}P(\bm x)
&=\langle\nabla_{\R^3} P(\bm x),\bm u\times\bm x\rangle\\ 
&=(u_2x_3-u_3x_2)g_1
+(u_3x_1-u_1x_3)g_2
+(u_1x_2-u_2x_1)g_3.
\label{eq:rotational-derivative-coordinate}
\end{aligned}
\end{equation}
Note that $g_1,g_2,g_3\in\Pi_{t-1}$, so $\Rot_{\bm u}P$ is also a polynomial in $\Pi_t$ for any fixed $\bm u\in\Sph$.

Rotational derivatives have been used extensively in spherical approximation theory.  Dai and Xu \cite{DaiXu2010,DaiXu2013} used these operators to establish direct and inverse theorems for polynomial approximation. In this paper, we use rotational derivatives to express the spherical gradient through scalar rotation averages and to control the Hessian by gradients of rotational derivatives. Since rotational differentiation preserves polynomial degree, Bernstein estimates can be applied repeatedly. The following lemma collects the identities and estimates needed for this argument.

\begin{lemma}
\label{lem:rotational-estimates}
Let $t\ge1$.  The following assertions hold.
\begin{enumerate}
\renewcommand{\labelenumi}{ \textup{(\roman{enumi})}}
\item For every $\bm u\in\Sph$ and $P\in\Pi_t$, we have $\Rot_{\bm u}P\in\Pi_t^0$. 

\item For every $P\in\Pi_t$ and $\bm x\in\Sph$,
\begin{equation}
\|\grad P(\bm x)\|_{\ell_2}
=2\int_{\Sph}|\Rot_{\bm u}P(\bm x)|\dd\sigma(\bm u).
\label{eq:axis-average-main}
\end{equation}

\item For every $Q\in\Pi_t$, we have
\begin{equation}
\int_{\Sph}\|\grad Q(\bm x)\|_{\ell_2}\dd\sigma(\bm x)
\le\frac{\pi t}{2}\int_{\Sph}|Q(\bm x)|\dd\sigma(\bm x).
\label{eq:spherical-L1-Bernstein}
\end{equation}

\item For every $P\in\Pi_t$,
\begin{equation}
\int_{\Sph}\|\Rot_{\bm u}P\|\dd\sigma(\bm u)
\le\frac{\pi t}{4}\|P\|.
\label{eq:rot-average-main}
\end{equation}

\item For every $P\in\Pi_t$ and $\bm x\in\Sph$,
\begin{equation}
\|\nabla^2 P(\bm x)\|_{\op}
\le2\int_{\Sph}
\|\grad\Rot_{\bm u}P(\bm x)\|_{\ell_2}\dd\sigma(\bm u).
\label{eq:hessian-axis-average}
\end{equation}
\end{enumerate}
\end{lemma}

\begin{proof}
See Appendix~\ref{pf-lem:rotational-estimates}.
\end{proof}

The following lemma shows that, for fixed polynomial degree, the supremum norm of the spherical gradient and the norms of all rotational derivatives are uniformly controlled by $\|Q\|$.

\begin{lemma}
\label{lem:finite-dimensional-remainder-bounds}
For every $t\ge1$, there is a constant $L_t\ge1$ such that for every $Q\in\Pi_t$,
\[
\sup_{\bm x\in\Sph}\|\grad Q(\bm x)\|_{\ell_2}
\le L_t\|Q\|,
\qquad
\sup_{\bm u\in\Sph}\|\Rot_{\bm u}Q\|\le L_t\|Q\|.
\]
\end{lemma}

\begin{proof}
If $\|Q\|=0$, its continuous spherical gradient vanishes, so $Q$ is constant on the connected sphere.  Thus $\|\cdot\|$ is a norm on $\Pi_t^0$.  The two maps
\[
Q\longmapsto \sup_{\bm x\in\Sph}\|\grad Q(\bm x)\|_{\ell_2},
\qquad
Q\longmapsto \sup_{\bm u\in\Sph}\|\Rot_{\bm u}Q\|
\]
are finite seminorms and hence continuous on this finite-dimensional space.  The first vanishes only on constants.  If the second vanishes, each $\Rot_{\bm u}Q$ is constant and has mean zero by Lemma~\ref{lem:rotational-estimates}\textup{(i)}, so it is zero. Equation~\eqref{eq:axis-average-main} then gives $\grad Q=\bm0$. Thus both are norms on $\Pi_t^0$.  By equivalence of norms \cite{Robinson2020}, there is a constant $L_t\ge1$ such that for every $Q\in\Pi_t^0$,
\[
\sup_{\bm x\in\Sph}\|\grad Q(\bm x)\|_{\ell_2}
\le L_t\|Q\|,
\qquad
\sup_{\bm u\in\Sph}\|\Rot_{\bm u}Q\|\le L_t\|Q\|.
\]
For an arbitrary $Q\in\Pi_t$, apply these inequalities to $Q-\int_{\Sph}Q\,\dd\sigma\in\Pi_t^0$. Subtracting a constant changes neither the spherical gradient nor any rotational derivative, and also leaves $\|Q\|$ unchanged. Hence both inequalities hold for every $Q\in\Pi_t$.

\end{proof}

\section{Area-regular partitions via the Ro\c{s}ca-Plonka map}
\label{sec:map}

This section supplies the geometric input for the sampling estimates. We construct $N=6q^2$ equal-area cells using the Ro\c{s}ca-Plonka map and choose their Karcher centers.  The first-variation identity provides the cancellation needed for quadrature, while Lemma~\ref{lem:karcher-localization} bounds the cell radius about each center and its second distance moment.

\subsection{Area-regular partitions}
The formulas in this subsection give the normalized-coordinate form of the explicit equal-area cubical projection of Ro\c{s}ca and Plonka~\cite[Sections~3-5]{RoscaPlonka2011}.  On the parameter square $[-1,1]^2$, if $|r|\le |s|$ and $s\ne0$, define
\begin{align*}
T_1(s,r)
&=2^{1/4}s
\frac{\sqrt2\cos\frac{\pi r}{12s}-1}{\sqrt{\sqrt2-\cos\frac{\pi r}{12s}}},
\quad T_2(s,r)
=2^{1/4}s
\frac{\sqrt2\sin\frac{\pi r}{12s}}{\sqrt{\sqrt2-\cos\frac{\pi r}{12s}}}.
\end{align*}
If $|s|\le|r|$ and $r\ne0$, use the symmetric formula
\begin{align*}
T_1(s,r)
&=2^{1/4}r
\frac{\sqrt2\sin\frac{\pi s}{12r}}
{\sqrt{\sqrt2-\cos\frac{\pi s}{12r}}},
\quad T_2(s,r)
=2^{1/4}r
\frac{\sqrt2\cos\frac{\pi s}{12r}-1}
{\sqrt{\sqrt2-\cos\frac{\pi s}{12r}}}.
\end{align*}
At the origin we set $T_1(0,0)=T_2(0,0)=0$. 
The north-face map $F:[-1,1]^2\to \mathcal{T}_1$ is defined directly by
\begin{equation}
F(s,r)=\left(
T_1\sqrt{1-\frac{T_1^2+T_2^2}{4}},
T_2\sqrt{1-\frac{T_1^2+T_2^2}{4}},
1-\frac{T_1^2+T_2^2}{2}
\right),
\label{eq:direct-face-map}
\end{equation}
where $T_1=T_1(s,r)$, $T_2=T_2(s,r)$, and
\begin{equation}\label{xueq42}
\mathcal{T}_1=\{(x,y,z)\in \Sph : z\geq |x|, z\geq |y|\} \subset\Sph.	
\end{equation}
The map is continuous on the closed square and is smooth in each of the open regions $|r|<|s|$ and $|s|<|r|$ inside $(-1,1)^2$. The set $\{(s,r):|s|=|r|\}$, where the two formulas meet, is called the formula seam.  These properties are part of the construction in \cite[Sections~3-5]{RoscaPlonka2011}.

For the six faces, choose the rotations
\[
\begin{aligned}
\mathcal R_1(x,y,z)&=(x,y,z),&
\mathcal R_2(x,y,z)&=(x,-y,-z),\\
\mathcal R_3(x,y,z)&=(z,y,-x),&
\mathcal R_4(x,y,z)&=(-z,y,x),\\
\mathcal R_5(x,y,z)&=(x,z,-y),&
\mathcal R_6(x,y,z)&=(x,-z,y).
\end{aligned}
\]
For $\ell=1,\ldots,6$, set 
\begin{equation*}
F_\ell=\mathcal R_\ell\circ F
\quad\text{and}\quad
\mathcal T_\ell=F_\ell([-1,1]^2)\subset\Sph.	
\end{equation*}
Note that $\mathcal T_\ell$, $\ell=1,\ldots,6$, are the six regions on $\Sph$ where $z,-z,x,-x,y,-y$ equals $\max\{|x|,|y|,|z|\}$, respectively. Every point of $\Sph$ has such a coordinate, so these regions cover the sphere.  Their interiors are pairwise disjoint, and overlaps occur only when two or more coordinates tie in absolute value.

Fix an integer $q\ge1$.  Define the closed parameter squares by
\[
Q_{a,b}=
\left[-1+\frac{2(a-1)}q,-1+\frac{2a}q\right]
\times
\left[-1+\frac{2(b-1)}q,-1+\frac{2b}q\right],
\qquad 1\le a,b\le q.
\]
Order the images by $i=(\ell-1)q^2+(a-1)q+b$ and define
\begin{equation}\label{xueq81}
C_i=F_\ell(Q_{a,b}),
\qquad 1\le i\le N=6q^2,
\end{equation}
so each cell is the closed spherical quadrilateral consisting of its interior and all four boundary curves. The interiors of these cells are pairwise disjoint.  Adjacent cells share their common boundary, and these boundaries have zero $\sigma$-measure. The equal-area property of the face map~\cite[Sections~3-5]{RoscaPlonka2011}, together with rotational invariance of $\sigma$,  gives
\[
\Sph=\bigcup_{i=1}^{N}C_i,
\qquad
\sigma(C_i)=\int_{C_i}\dd\sigma=\frac1N, \quad\forall i,
\quad\text{and}\quad
\sigma(C_i\cap C_j)=0\quad\forall i\ne j.
\]
Consequently, for every integrable $f$, the average of $f$ on $C_i$ is $N\int_{C_i}f\dd\sigma$.  Moreover,
\[
\sum_{i=1}^N\int_{C_i}f\dd\sigma=\int_{\Sph}f\dd\sigma.
\]

%

\begin{figure}[!t]
\centering
\includegraphics[width=.6\textwidth]
{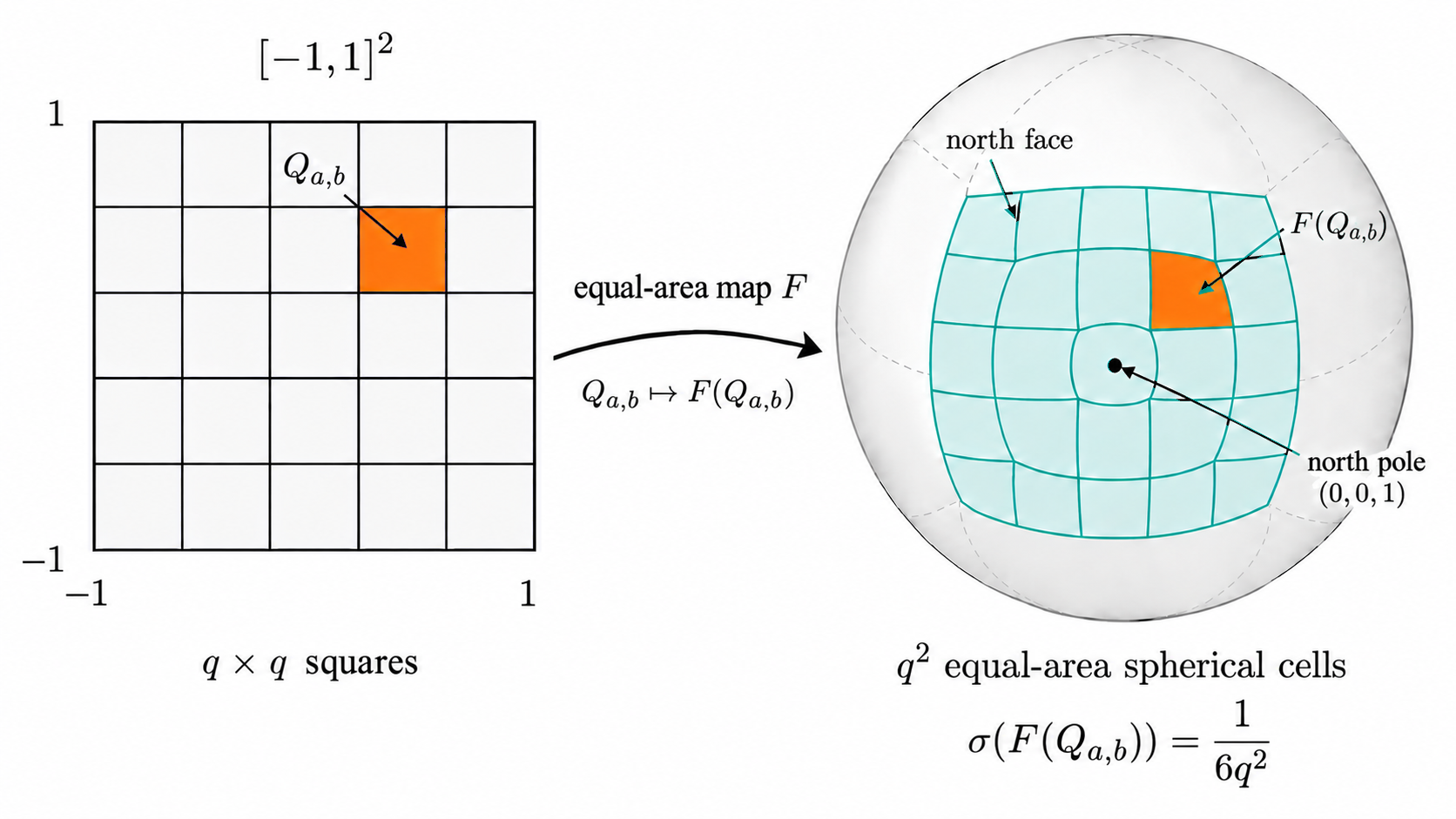}
\caption{The equal-area correspondence on one cubed-sphere face.  A uniform
$q\times q$ subdivision of $[-1,1]^2$ is carried by the corresponding
face map to $q^2$
equal-measure spherical cells.  Over all six faces there are $N=6q^2$
cells, each satisfying $\sigma(C_i)=1/N$.  The highlighted planar
square and spherical cell form one corresponding pair.}
\label{fig:equal-measure-face-map}
\end{figure}

\begin{figure}[!t]
\centering
\includegraphics[
width=.3\textwidth,
trim=110 80 110 100,
clip
]{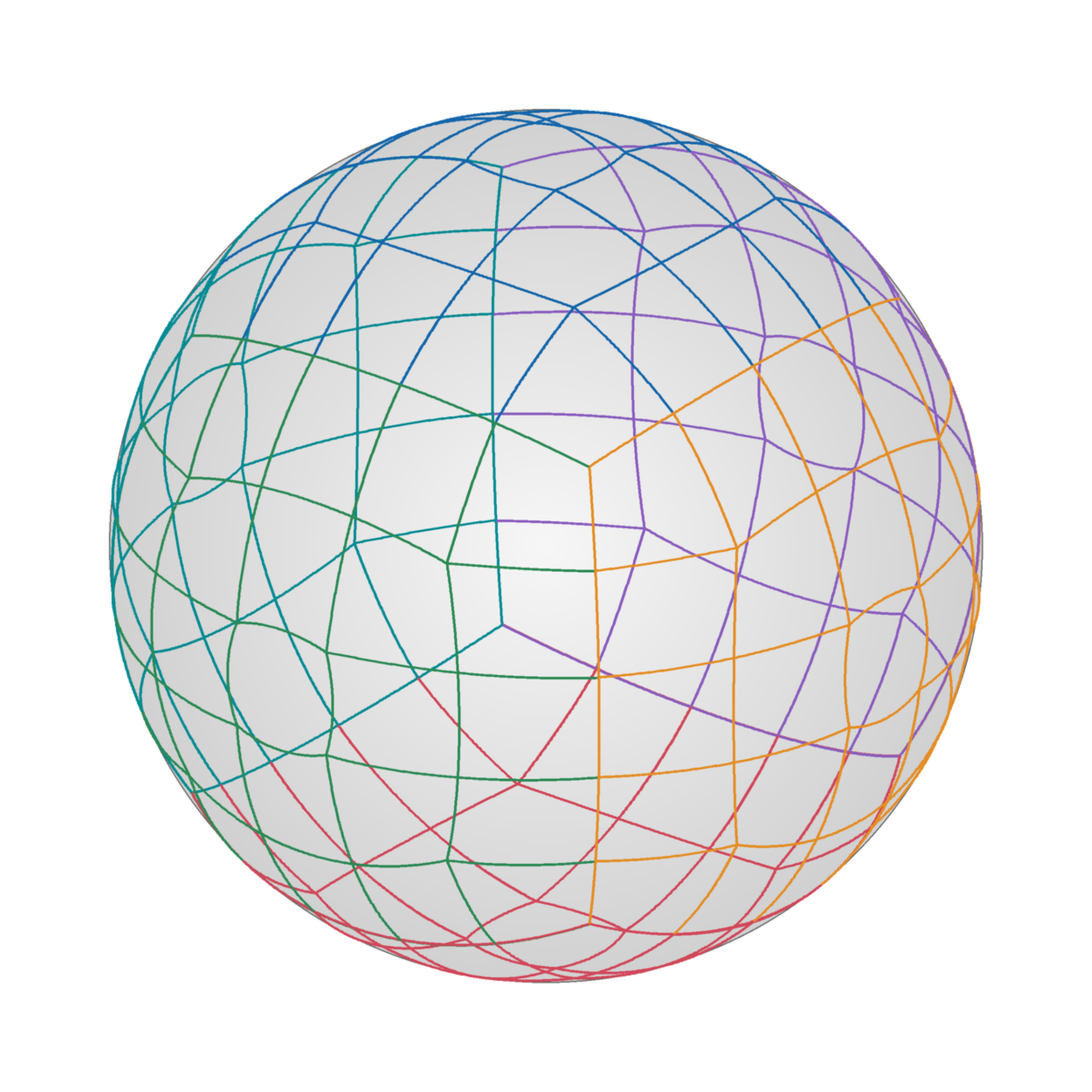}
\caption{The Ro\c{s}ca-Plonka cubed-sphere grid with $q=5$.  Each 
grid square corresponds to a spherical cell of normalized measure $1/N$.}
\label{fig:grid}
\end{figure}

\subsection{Cellwise Karcher centering}
\label{sec:karcher-centering}

\begin{celldefinition}
\label{def:cell-karcher-center}
Fix an integer $q\ge1$, set $N=6q^2$, and let
$C_1,\ldots,C_N$ be the equal-area cubed-sphere cells constructed in \eqref{xueq81}.
For each $i$, define the cell energy
\[
\mathcal E_i(\bm z)
:=\int_{C_i}\dist(\bm z,\bm y)^2\dd\sigma(\bm y),
\qquad \bm z\in\Sph.
\]
A minimizer of $\mathcal E_i$ is called a \emph{Karcher center of the
cubed-sphere cell $C_i$}, that is,
the Karcher mean (or Riemannian center of mass) of
the normalized area measure on $C_i$ (see also  \cite{Afsari2011,Karcher1977}).

\end{celldefinition}

\begin{remark}
\label{rem:cell-karcher-center}
Each closed cell $C_i$ lies in a geodesic ball of radius less than $\pi/2$.  To see this, first consider the closed north face $\mathcal T_1$ defined in equation~\eqref{xueq42}.  Every $\bm y=(y_1,y_2,y_3)\in\mathcal T_1$ satisfies $y_3\ge|y_1|$, $y_3\ge|y_2|$, and hence $1=y_1^2+y_2^2+y_3^2\le3y_3^2$.  Thus, with $\bm e_3=(0,0,1)$,
\[
\dist(\bm e_3,\bm y)=\arccos y_3
\le\arccos\frac1{\sqrt3}<\frac\pi2,
\quad \forall \bm y\in \mathcal T_1.
\]
Since $\arccos(1/\sqrt3)<\pi/3$, we have $\mathcal T_1\subset B(\bm e_3,\pi/3)$. Rotating this inclusion gives, for every cell on face $\ell$,
\[
C_i\subset\mathcal T_\ell
\subset B\left(\mathcal R_\ell\bm e_3,\frac\pi3\right).
\]
Afsari's global uniqueness theorem for the Riemannian $L^2$ center of mass \cite[Theorem~2.1]{Afsari2011} therefore gives a unique Karcher center $\bm c_i$, which lies in the same ball.  Consequently, $\dist(\bm c_i,\bm y)<2\pi/3<\pi$ for $\bm y\in C_i$, so the energy is smooth at $\bm c_i$.  Its vanishing gradient, equivalently the first-variation formula in \cite[equation~(2.5)]{Afsari2011}, gives
\begin{equation}
\int_{C_i}\xi_{\bm c_i}(\bm y)\dd\sigma(\bm y)=\bm0,
\label{eq:karcher-first-variation}
\end{equation}
where $\xi_{\bm c_i}(\bm y)$ is defined in equation~\eqref{eq:tangent-displacement}.

\end{remark}
The next lemma quantifies the radius and second distance moment about these centers.  These two estimates will bound all the moments used in Section~\ref{sec:sample}.

\begin{lemma}
\label{lem:karcher-localization}
Let $q\geq 6$ be an integer and let $N=6q^2$. Let $C$ be one of the equal-area cells constructed in \eqref{xueq81}. 
Let $\bm c$ be its Karcher center. 
\begin{enumerate}
\item[\rm (i)] We have $C \subset B\left(\bm c,\frac{15}{\sqrt{8N}}\right)$.

\item[\rm (ii)] We have
\begin{equation}
\int_{C}\dist(\bm c,\bm y)^2\dd\sigma(\bm y)
<\frac8{3N^2}.
\label{eq:centered-cell-quarter}
\end{equation}
\end{enumerate}

\end{lemma}

\begin{proof}
See  Appendix~\ref{app:karcher-localization-proof}.
\end{proof}

\section{Sampling at Karcher centers}
\label{sec:sample}

This section converts the cell geometry into bounds for gradient sampling and quadrature at the Karcher centers.  An ordered recurrence expresses both errors in terms of the distance-moment tails $\beta_1$ and $\beta_2$.  The quadrature bound then gives an explicit bound for the gradient sampling constant, which also controls sampled Hessians.

\subsection{Distance moments and the radial recurrence}
Let $t\ge1$ and $q\ge1$ be integers.  Set $N=6q^2$ and let $C_1,\ldots,C_N$ be the equal-area cubed-sphere cells constructed in \eqref{xueq81}. For each $i\in\{1,2,\ldots,N\}$, let $\bm c_i$ be the unique Karcher center of $C_i$.  For integers $m\ge1$ and $j\ge1$, define
\begin{equation}
\mu_m=N\max_{1\leq i\leq N}\int_{C_i}
\dist(\bm c_i,\bm y)^m\dd\sigma(\bm y),
\qquad
\beta_j=\sum_{m=j}^{\infty}\frac{\mu_m}{m!}
\left(\frac{\pi t}{2}\right)^m.
\label{eq:centered-tails}
\end{equation}
The factor $N$ makes each integral in $\mu_m$ a cell average. Since $\sigma(C_i)=1/N$ and $\dist(\bm c_i,\bm y)\le\pi$, we have
\[
0\le\mu_m
=N\max_{1\le i\le N}\int_{C_i}
\dist(\bm c_i,\bm y)^m\,\dd\sigma(\bm y)
\le\pi^m.
\]
Consequently, for every fixed $t\ge1$ and every integer $j\ge1$,
\[
0\le \beta_j
\le\sum_{m=j}^{\infty}\frac{1}{m!}
\left(\frac{\pi^2t}{2}\right)^m
\le\exp\left(\frac{\pi^2t}{2}\right)<\infty,
\]
so each series $\beta_j$ converges by comparison with the exponential series.

To bound the sampled gradient norm uniformly over polynomials, define
\begin{equation}\label{xueq12}
\mathcal{R}_{\mathrm{max}}=\sup_{H\in\Pi_t^0,\|H\|=1}
\frac1N\sum_{i=1}^N\|\grad H(\bm c_i)\|_{\ell_2}.
\end{equation}
The quantity $\mathcal{R}_{\mathrm{max}}$ is finite by the first estimate in Lemma~\ref{lem:finite-dimensional-remainder-bounds}.  For $Q\in\Pi_t$, an integer $k\ge0$, and $0\le s\le1$, define
\[
W_k^Q(s)=\sum_{i=1}^N\int_{C_i}
\dist(\bm c_i,\bm y)^k\cdot 
\|\grad Q(\gamma_{\bm c_i,\bm y}(1-s))\|_{\ell_2}\dd\sigma(\bm y),
\]
where the geodesic $\gamma_{\bm c_i,\bm y}(\cdot )$ is defined in equation~\eqref{eq:radial-geodesic}. As usual, the radial weight $\dist(\bm c_i,\bm y)^k$ equals $1$ when $k=0$ and equals $0$ at $\bm y=\bm c_i$ when $k\ge1$.

The quantities $W_k^Q$ compare gradients along the geodesics from the cells to their centers while retaining the $k$-th power of the distance. In particular, $W_0^Q(0)=\|Q\|$ and $W_0^Q(1)=\frac1N\sum_{i=1}^N\|\grad Q(\bm c_i)\|_{\ell_2}$. The next lemma raises the distance power by one at each rotational differentiation and bounds the endpoint by $\mu_m$. Iterating it will give the gradient error involving $\beta_1$ in Lemma~\ref{lem:centered-gradient-closure} and, after the linear Taylor term cancels, the quadrature error involving $\beta_2$ in Lemma~\ref{lem:centered-value-closure}.

\begin{lemma}
\label{lem:centered-radial-W-estimates}
For every $Q\in\Pi_t$, every integer $k\ge0$, and $0\le s\le1$,
\begin{equation} 
\left|W_k^Q(s)-W_k^Q(1)\right|
\le2\int_s^1\int_{\Sph}
W_{k+1}^{\Rot_{\bm u}Q}(v)
\dd\sigma(\bm u)\dd v,
\label{eq:ordered-radial-recurrence}
\end{equation}
where the rotational derivative $\Rot_{\bm u}$ is defined in equation~\eqref{eq:rotational-derivative}. Moreover,  for $m\ge1$ and $Q\in\Pi_t^0$,
\begin{equation}\label{xueq9}
W_m^Q(1)
\le\mu_m\mathcal{R}_{\mathrm{max}}\|Q\|.
\end{equation}
Here $\mathcal{R}_{\mathrm{max}}$ and $\mu_m$ are defined in equation~\eqref{xueq12} and equation~\eqref{eq:centered-tails}, respectively.

\end{lemma}

\begin{proof}
Note that $\gamma_{\bm c_i,\bm y}(0)=\bm c_i$ and
\begin{equation}\label{xueq15}
\begin{aligned}
&\left|W_k^Q(s)-W_k^Q(1)\right|\\
&\leq 	\sum_{i=1}^N\int_{C_i}
\dist(\bm c_i,\bm y)^k\cdot 
\bigg|\|\grad Q(\gamma_{\bm c_i,\bm y}(1-s))\|_{\ell_2}-\|\grad Q(\gamma_{\bm c_i,\bm y}(0))\|_{\ell_2}\bigg|\dd\sigma(\bm y).
\end{aligned}	
\end{equation}
Applying Lemma~\ref{lem:spherical-gradient-variation-curve} to $\bm z(v):=\gamma_{\bm c_i,\bm y}(1-v)$ on the interval $[s,1]$ gives
\begin{equation}\label{xueq16}
\left|
\|\grad Q(\bm z(s))\|_{\ell_2}-\|\grad Q(\bm z(1))\|_{\ell_2}
\right|
\le \dist(\bm c_i,\bm y)\int_s^1\|\nabla^2Q(\bm z(v))\|_{\op}\dd v.
\end{equation}
Here, we use the fact that the speed of $\bm z(v)$ is $\dist(\bm c_i,\bm y)$. Substituting equation~\eqref{xueq16} into equation~\eqref{xueq15}, we obtain
\begin{equation}\label{xueq1}
\begin{aligned}
\left|W_k^Q(s)-W_k^Q(1)\right|
&\le
\sum_{i=1}^N\int_{C_i}\dist(\bm c_i,\bm y)^{k+1}
\int_s^1\|\nabla^2Q(\gamma_{\bm c_i,\bm y}(1-v))\|_{\op}
\dd v\,\dd\sigma(\bm y).
\end{aligned}
\end{equation}
By the Hessian estimate in equation~\eqref{eq:hessian-axis-average}, we have
\[
\|\nabla^2Q(\gamma_{\bm c_i,\bm y}(1-v))\|_{\op}\leq 2\int_{\Sph}
\|\grad\Rot_{\bm u}Q(\gamma_{\bm c_i,\bm y}(1-v))\|_{\ell_2}
\dd\sigma(\bm u).
\]
Substituting this estimate into equation~\eqref{xueq1} and using the definition of $W_{k+1}^{\Rot_{\bm u}Q}(v)$, we obtain
\[
|W_k^Q(s)-W_k^Q(1)|
\le2\int_s^1\int_{\Sph}
W_{k+1}^{\Rot_{\bm u}Q}(v)
\dd\sigma(\bm u)\dd v.
\]
Finally, if $m\ge1$ and $Q\in\Pi_t^0$, then
\[
\begin{aligned}
W_m^Q(1)
&=\sum_{i=1}^N\|\grad Q(\bm c_i)\|_{\ell_2}
\int_{C_i}\dist(\bm c_i,\bm y)^m\dd\sigma(\bm y)\le
\frac{\mu_m}{N}
\sum_{i=1}^N\|\grad Q(\bm c_i)\|_{\ell_2}\le\mu_m\mathcal{R}_{\mathrm{max}}\|Q\|.
\end{aligned}
\]
Here the first equality uses $\gamma_{\bm c_i,\bm y}(0)=\bm c_i$ from equation~\eqref{eq:radial-geodesic}. The first inequality uses the definition of $\mu_m$ in equation~\eqref{eq:centered-tails}, and the last inequality is exactly the definition of $\mathcal{R}_{\mathrm{max}}$.
\end{proof}

\subsection{Gradient sampling error}  
The following lemma compares the average sampled gradient norm with its integral over the sphere.

\begin{lemma}
\label{lem:centered-gradient-closure}
For every $H\in\Pi_t^0$,
\begin{align}
\left|\frac1N\sum_{i=1}^N\|\grad H(\bm c_i)\|_{\ell_2}-\|H\|\right|
\le \mathcal{R}_{\mathrm{max}} \beta_1 \|H\|,
\label{eq:centered-closure}
\end{align}
where $\mathcal{R}_{\mathrm{max}}$ and $\beta_1$ are defined in equation~\eqref{xueq12} and equation~\eqref{eq:centered-tails}, respectively.

\end{lemma}

\begin{proof}
Note that
\[
W_0^H(0)=\|H\|,
\qquad
W_0^H(1)=\frac1N\sum_{i=1}^N\|\grad H(\bm c_i)\|_{\ell_2},
\]	
by the definition of $W_0^H$, by $\gamma_{\bm c_i,\bm y}(1)=\bm y$ and $\gamma_{\bm c_i,\bm y}(0)=\bm c_i$ from equation~\eqref{eq:radial-geodesic}, and by $\sigma(C_i)=1/N$.  Hence it is enough to prove
\[
\left|W_0^H(0)-W_0^H(1)\right|
\le \mathcal{R}_{\mathrm{max}}\beta_1\|H\|.
\]
For axes $\bm u_1,\ldots,\bm u_m$, set
\[
H_{0}=H,
\qquad
H_{j}=\Rot_{\bm u_j}H_{j-1},
\]
and let
\[
\Delta_m
=\{0\le s_1\le\cdots\le s_m\le1\}.
\]
A direct calculation gives
\begin{equation}
\int_{\Delta_m}\prod_{j=1}^{m}\dd s_j
=\frac1{m!}.
\label{eq:value-weighted-simplex-proof2}
\end{equation}
For $m\ge1$, set
\begin{equation}
\begin{aligned}
E_m&:=
2^m\int_{\Delta_m}\int_{(\Sph)^m}
W_m^{H_m}(1)
\prod_{j=1}^m\dd\sigma(\bm u_j)
\prod_{j=1}^m\dd s_j,\\
R_m&:=
2^m\int_{\Delta_m}\int_{(\Sph)^m}
W_m^{H_m}(s_m)
\prod_{j=1}^m\dd\sigma(\bm u_j)
\prod_{j=1}^m\dd s_j
\end{aligned}
\end{equation}
For $m\geq 1$ and $(s_1,\ldots,s_{m})\in\Delta_{m}$, equation~\eqref{eq:ordered-radial-recurrence} in Lemma~\ref{lem:centered-radial-W-estimates} gives
\[
\begin{aligned}
W_{m}^{H_m}(s_{m})
&\le W_{m}^{H_m}(1)+
2\int_{s_{m}}^1\int_{\Sph}
W_{m+1}^{H_{m+1}}(s_{m+1})
\dd\sigma(\bm u_{m+1})\dd s_{m+1}.
\end{aligned}
\]
Multiplying by the prefactor and measures in the definition of $R_{m}$ shows that for each $m\geq 1$,
\begin{equation}\label{xueq6}
R_{m}\le E_{m}+R_{m+1}.
\end{equation}
Applying equation~\eqref{eq:ordered-radial-recurrence} with $k=0$ and $s=0$ gives
\begin{equation}\label{xueq3}
\left|W_0^H(0)-W_0^H(1)\right|
\le
2\int_0^1\int_{\Sph}W_1^{H_1}(s_1)
\dd\sigma(\bm u_1)\dd s_1
=R_1\leq E_1+R_2,
\end{equation}
where the last inequality follows from equation~\eqref{xueq6}. Then, repeatedly using equation~\eqref{xueq6}, we obtain that for any $M\geq 1$,
\begin{equation}\label{xueq5}
\left|W_0^H(0)-W_0^H(1)\right|
\le \sum_{m=1}^ME_m+R_{M+1}.
\end{equation}

We next estimate each $E_m$ and $R_{M+1}$. Note that $H_m\in\Pi_t^0$ by Lemma~\ref{lem:rotational-estimates}  \textup{(i)}, applied successively.  Hence the endpoint estimate in Lemma~\ref{lem:centered-radial-W-estimates} gives
\[
W_m^{H_m}(1)
\le\mu_m\mathcal{R}_{\mathrm{max}}\|H_m\|.
\]
Using equation~\eqref{eq:value-weighted-simplex-proof2}, we get
\[
\begin{aligned}
E_m
&\le
\frac{2^m\mu_m\mathcal{R}_{\mathrm{max}}}{m!}
\int_{(\Sph)^m}\|H_m\|
\prod_{j=1}^m\dd\sigma(\bm u_j) \le \mathcal{R}_{\mathrm{max}}\frac{\mu_m}{m!}
\left(\frac{\pi t}{2}\right)^m\|H\|.
\end{aligned}
\]
Here the second inequality follows by applying equation~\eqref{eq:rot-average-main} successively to $H_j=\Rot_{\bm u_j}H_{j-1}$.

For the remaining term $R_{M+1}$, Remark~\ref{rem:cell-karcher-center} gives $\dist(\bm c_i,\bm y)<\pi$.  We then use the first estimate in Lemma~\ref{lem:finite-dimensional-remainder-bounds}. Let $L_t$ be the constant in Lemma~\ref{lem:finite-dimensional-remainder-bounds}. Since the cells have total $\sigma$-measure $1$,
\[
W_{M+1}^{H_{M+1}}(s_{M+1})
\le \pi^{M+1}L_t\|H_{M+1}\|
\le \pi^{M+1}L_t^{M+2}\|H\|.
\]
The last inequality uses the second estimate in Lemma~\ref{lem:finite-dimensional-remainder-bounds} $M+1$ times.  Therefore, using equation~\eqref{eq:value-weighted-simplex-proof2}, we have
\[
R_{M+1}
\le
2^{M+1}\frac1{(M+1)!}\pi^{M+1}L_t^{M+2}\|H\|
=
L_t\frac{(2\pi L_t)^{M+1}}{(M+1)!}\|H\|,
\]
which tends to zero as $M\to\infty$.  Letting $M\to\infty$ in equation~\eqref{xueq5} proves
\[
\left|W_0^H(0)-W_0^H(1)\right|
\le \mathcal{R}_{\mathrm{max}}\sum_{m=1}^{\infty}\frac{\mu_m}{m!}
\left(\frac{\pi t}{2}\right)^m\|H\|
=\mathcal{R}_{\mathrm{max}}\beta_1\|H\|.
\]
This completes the proof.

\end{proof}

\subsection{Quadrature error at Karcher centers}

The next lemma bounds the quadrature error at the Karcher centers of the cells. Since each Karcher center $\bm c_i$ minimizes the integral of squared geodesic distance over $C_i$, its first-order optimality condition gives
\[
\int_{C_i}\xi_{\bm c_i}(\bm y)\,\dd\sigma(\bm y)=\bm0.
\]
Here $\xi_{\bm c_i}(\bm y)$ is the tangent displacement defined in \eqref{eq:tangent-displacement}. Consequently, the linear term $\langle\grad Q(\bm c_i),\xi_{\bm c_i}(\bm y)\rangle$ in the geodesic Taylor formula vanishes upon integration over $C_i$. The remaining integral remainder is bounded by $\dist(\bm c_i,\bm y)^2$ times a bound for the spherical Hessian along the corresponding geodesic. Thus the quadrature error is controlled by second moments of the distances from the cell centers.

\begin{lemma}
\label{lem:centered-value-closure}
For every $Q\in\Pi_t$,
\begin{equation}
\frac1N\sum_{i=1}^N\left|Q(\bm c_i)-N\int_{C_i}Q\dd\sigma\right|
\le\frac{2\beta_2}{\pi t}\cdot \mathcal{R}_{\mathrm{max}}\cdot \|Q\|,
\label{eq:centered-value-closure}
\end{equation}
where $\mathcal{R}_{\mathrm{max}}$ and $\beta_2$ are defined in equation~\eqref{xueq12} and equation~\eqref{eq:centered-tails}, respectively.

\end{lemma}

\begin{proof}
By Remark~\ref{rem:cell-karcher-center}, $\dist(\bm c_i,\bm y)<\pi$ for $\bm y\in C_i$.  Hence Lemma~\ref{lem:geodesic-taylor-second-order}  gives
\begin{equation}
\begin{aligned}
Q(\bm y)-Q(\bm c_i)
&=\langle\grad Q(\bm c_i),\xi_{\bm c_i}(\bm y)\rangle\\
&\quad+\int_0^1(1-\tau)
\left\langle
\nabla^2 Q(\gamma_{\bm c_i,\bm y}(\tau))
\frac{\dd\gamma_{\bm c_i,\bm y}}{\dd\tau}(\tau),
\frac{\dd\gamma_{\bm c_i,\bm y}}{\dd\tau}(\tau)
\right\rangle\dd\tau.
\end{aligned}
\label{eq:value-taylor-local-proof}
\end{equation}
Integrating equation~\eqref{eq:value-taylor-local-proof} with respect to $\bm y$ over $C_i$, the first-order term vanishes by the Karcher first-variation identity equation~\eqref{eq:karcher-first-variation}:
\begin{equation}
\int_{C_i}
\langle\grad Q(\bm c_i),\xi_{\bm c_i}(\bm y)\rangle\dd\sigma(\bm y)
=
\left\langle
\grad Q(\bm c_i),
\int_{C_i}\xi_{\bm c_i}(\bm y)\dd\sigma(\bm y)
\right\rangle
=\left\langle
\grad Q(\bm c_i),
\bm0
\right\rangle
=0.
\label{eq:value-first-order-vanishes-proof}
\end{equation}
Combining equation~\eqref{eq:value-taylor-local-proof} and equation~\eqref{eq:value-first-order-vanishes-proof}, we get
\begin{equation}
\begin{aligned}
&\frac{1}{N}Q(\bm c_i)-\int_{C_i}Q\dd\sigma\\
&=-\int_{C_i}\int_0^1(1-\tau)
\left\langle
\nabla^2 Q(\gamma_{\bm c_i,\bm y}(\tau))
\frac{\dd\gamma_{\bm c_i,\bm y}}{\dd\tau}(\tau),
\frac{\dd\gamma_{\bm c_i,\bm y}}{\dd\tau}(\tau)
\right\rangle\dd\tau\dd\sigma(\bm y).
\end{aligned}
\label{eq:value-remainder-identity-proof}
\end{equation}
Taking absolute values and using both the operator norm and the constant-speed property in equation~\eqref{eq:radial-geodesic}, namely
\[
\left\|\frac{\dd\gamma_{\bm c_i,\bm y}}{\dd\tau}(\tau)\right\|_{\ell_2}
=\dist(\bm c_i,\bm y),
\]
we obtain from equation~\eqref{eq:value-remainder-identity-proof} that
\begin{equation}
\frac1N\left|
Q(\bm c_i)-N\int_{C_i}Q\dd\sigma
\right|
\le
\int_{C_i}\int_0^1
(1-\tau)\dist(\bm c_i,\bm y)^2
\|\nabla^2Q(\gamma_{\bm c_i,\bm y}(\tau))\|_{\op}
\dd\tau\dd\sigma(\bm y).
\label{eq:value-local-hessian-proof}
\end{equation}
Using the change of variables $v_1=1-\tau$ gives
\begin{equation}
\begin{aligned}
\frac1N\left|
Q(\bm c_i)-N\int_{C_i}Q\dd\sigma
\right|
&\le
\int_{C_i}\int_0^1
v_1\dist(\bm c_i,\bm y)^2
\|\nabla^2Q(\gamma_{\bm c_i,\bm y}(1-v_1))\|_{\op}
\dd v_1\dd\sigma(\bm y).
\end{aligned}
\label{eq:value-local-vone-proof}
\end{equation}
By the Hessian estimate in equation~\eqref{eq:hessian-axis-average},
\begin{equation}
\|\nabla^2Q(\gamma_{\bm c_i,\bm y}(1-v_1))\|_{\op}
\le
2\int_{\Sph}
\|\grad\Rot_{\bm u}Q(\gamma_{\bm c_i,\bm y}(1-v_1))\|_{\ell_2}
\dd\sigma(\bm u).
\label{eq:value-hessian-axis-proof}
\end{equation}
Substituting equation~\eqref{eq:value-hessian-axis-proof} into equation~\eqref{eq:value-local-vone-proof} and summing over $i$, we obtain
\begin{equation}
\begin{aligned}
&\frac1N\sum_{i=1}^N
\left|Q(\bm c_i)-N\int_{C_i}Q\dd\sigma\right|\\
&\le
2\int_{\Sph}\int_0^1v_1
\sum_{i=1}^N\int_{C_i}
\dist(\bm c_i,\bm y)^2 \cdot
\|\grad\Rot_{\bm u}Q(\gamma_{\bm c_i,\bm y}(1-v_1))\|_{\ell_2}
\dd\sigma(\bm y)\dd v_1\dd\sigma(\bm u)\\
&=
2\int_{\Sph}\int_0^1v_1
W_2^{\Rot_{\bm u}Q}(v_1)
\dd v_1\dd\sigma(\bm u),
\end{aligned}
\label{eq:value-reduction-to-Wtwo-proof}
\end{equation}
where the last equality follows from the definition of $W_2^{\Rot_{\bm u}Q}$.

We next bound the last integral in equation~\eqref{eq:value-reduction-to-Wtwo-proof}. For axes $\bm u_1,\ldots,\bm u_{m-1}$, set
\[
Q_0=Q,
\qquad
Q_j=\Rot_{\bm u_j}Q_{j-1}.
\]
For $m\ge2$, let
\[
\Lambda_m
=\{(v_1,\ldots,v_{m-1}):0\le v_1\le\cdots\le v_{m-1}\le1\}.
\]
Integrating first over $v_2,\ldots,v_{m-1}$ gives
\begin{equation}
\int_{\Lambda_m}v_1\prod_{j=1}^{m-1}\dd v_j
=\int_0^1v_1\frac{(1-v_1)^{m-2}}{(m-2)!}\dd v_1
=\frac1{m!}.
\label{eq:value-weighted-simplex-proof}
\end{equation}
For $m\geq 2$, set
\begin{equation}\label{xueq7}
\begin{aligned}
\widehat{E}_m&:=
2^{m-1}\int_{\Lambda_m}\int_{(\Sph)^{m-1}}
v_1\cdot W_m^{Q_{m-1}}(1)
\prod_{j=1}^{m-1}\dd\sigma(\bm u_j)
\prod_{j=1}^{m-1}\dd v_j,\\
\widehat{R}_m&:=
2^{m-1}\int_{\Lambda_m}\int_{(\Sph)^{m-1}}
v_1\cdot W_m^{Q_{m-1}}(v_{m-1})
\prod_{j=1}^{m-1}\dd\sigma(\bm u_j)
\prod_{j=1}^{m-1}\dd v_j	
\end{aligned}	
\end{equation}
By a calculation similar to the one used in equation~\eqref{xueq6}, we obtain, for any $m\geq 2$,
\begin{equation}\label{xueq8}
\widehat{R}_m\leq \widehat{E}_m	+\widehat{R}_{m+1}.
\end{equation}
Equation~\eqref{eq:ordered-radial-recurrence} in Lemma~\ref{lem:centered-radial-W-estimates}, applied to $W_2^{Q_1}(v_1)$, gives
\[
W_2^{Q_1}(v_1)
\le
W_2^{Q_1}(1)
+2\int_{v_1}^1\int_{\Sph}
W_3^{Q_2}(v_2)\dd\sigma(\bm u_2)\dd v_2.
\]
Multiplying by $2v_1$ and integrating in $v_1$ and $\bm u_1$ gives 
\[
\frac1N\sum_{i=1}^N
\left|Q(\bm c_i)-N\int_{C_i}Q\dd\sigma\right|
\leq 2\int_{\Sph}\int_0^1v_1
W_2^{\Rot_{\bm u_1}Q}(v_1)
\dd v_1\dd\sigma(\bm u_1)
\le \widehat{E}_2+\widehat{R}_3.
\]
Here, the first inequality follows from equation~\eqref{eq:value-reduction-to-Wtwo-proof}. Repeatedly using equation~\eqref{xueq8}, we obtain, for any $M\ge2$,
\begin{equation} 
\frac1N\sum_{i=1}^N
\left|Q(\bm c_i)-N\int_{C_i}Q\dd\sigma\right|
\le\sum_{m=2}^M\widehat{E}_m+\widehat{R}_{M+1}.
\label{eq:value-finite-radial-expansion-proof}
\end{equation}
Using equation~\eqref{xueq9} in Lemma~\ref{lem:centered-radial-W-estimates}, the weighted volume identity in equation~\eqref{eq:value-weighted-simplex-proof}, and equation~\eqref{eq:rot-average-main}, we obtain
\[
\begin{aligned}
\widehat{E}_m
&\le
\frac{2^{m-1}\mu_m\mathcal{R}_{\mathrm{max}}}{m!}
\int_{(\Sph)^{m-1}}\|Q_{m-1}\|
\prod_{j=1}^{m-1}\dd\sigma(\bm u_j)\le
\frac{2\mathcal{R}_{\mathrm{max}}\mu_m}{\pi tm!}\left(\frac{\pi t}{2}\right)^m\|Q\|.
\end{aligned}
\]
Let $L_t$ be the constant in Lemma~\ref{lem:finite-dimensional-remainder-bounds}.  Since $\dist(\bm c_i,\bm y)<\pi$ by Remark~\ref{rem:cell-karcher-center}, the first estimate in Lemma~\ref{lem:finite-dimensional-remainder-bounds} gives
\[
W_{M+1}^{Q_M}(v_M)
\le \pi^{M+1}L_t\|Q_M\|
\le \pi^{M+1}L_t^{M+1}\|Q\|.
\]
The last inequality uses the second estimate in Lemma~\ref{lem:finite-dimensional-remainder-bounds} $M$ times.  Using equation~\eqref{eq:value-weighted-simplex-proof} with $m=M+1$, we obtain
\[
\widehat{R}_{M+1}
\le
2^M\frac1{(M+1)!}\pi^{M+1}L_t^{M+1}\|Q\|
=\frac{(2\pi L_t)^{M+1}}{2(M+1)!}\|Q\|,
\]
which tends to zero as $M\to\infty$.  Letting $M\to\infty$ in equation~\eqref{eq:value-finite-radial-expansion-proof} proves
\[
\frac1N\sum_{i=1}^N
\left|Q(\bm c_i)-N\int_{C_i}Q\dd\sigma\right|\le
\frac{2\mathcal{R}_{\mathrm{max}}}{\pi t}
\sum_{m=2}^{\infty}\frac{\mu_m}{m!}
\left(\frac{\pi t}{2}\right)^m\|Q\|
=
\frac{2\mathcal{R}_{\mathrm{max}}\beta_2}{\pi t}\|Q\|,
\]
where the last equality is the definition of $\beta_2$ in equation~\eqref{eq:centered-tails}.  This proves the lemma.
\end{proof}

\subsection{The sampling constant and Hessian bounds}
\begin{lemma}
\label{lem:centered-sampling-constant}
If $\beta_2<1$, then
\begin{align}
\mathcal{R}_{\mathrm{max}}&\le\frac1{1-\beta_2},
\label{eq:centered-sampling-constant}
\end{align}
where $\mathcal{R}_{\mathrm{max}}$ and $\beta_2$ are defined in equation~\eqref{xueq12} and equation~\eqref{eq:centered-tails}, respectively.
\end{lemma}

\begin{proof}
For $H\in\Pi_t^0$, write
\[
\mathcal A_H=\frac1N\sum_{i=1}^N
\|\grad H(\bm c_i)\|_{\ell_2}.
\]
Using equation~\eqref{eq:axis-average-main}, adding and subtracting the cell averages, and applying Lemma~\ref{lem:centered-value-closure} to $\Rot_{\bm u}H\in\Pi_t^0$, we obtain
\[
\begin{aligned}
\mathcal A_H
&=2\int_{\Sph}\frac1N\sum_{i=1}^N
|\Rot_{\bm u}H(\bm c_i)|\dd\sigma(\bm u)\\
&\le
2\int_{\Sph}\sum_{i=1}^N\int_{C_i}
|\Rot_{\bm u}H(\bm y)|\dd\sigma(\bm y)\dd\sigma(\bm u)
+\frac{4\mathcal R_{\mathrm{max}}\beta_2}{\pi t}
\int_{\Sph}\|\Rot_{\bm u}H\|\dd\sigma(\bm u)\\
&\le \|H\|+\mathcal R_{\mathrm{max}}\beta_2\|H\|.
\end{aligned}
\]
The last line uses equations~\eqref{eq:axis-average-main} and~\eqref{eq:rot-average-main}.  Taking the supremum over $\|H\|=1$ gives
\[
\mathcal R_{\mathrm{max}}\le1+\mathcal R_{\mathrm{max}}\beta_2.
\]
Since $\mathcal R_{\mathrm{max}}$ is finite and $\beta_2<1$, rearranging proves equation~\eqref{eq:centered-sampling-constant}.
\end{proof}

To control the moving points, we also need to bound the average of the operator norms of the spherical Hessians at the sampled points. The next lemma obtains this bound from any gradient sampling estimate by using rotational derivatives, which remain scalar polynomials of degree at most $t$.  Thus no separate sampling space for matrix-valued Hessians is needed.

\begin{lemma}
\label{lem:hessian-sampling-transfer}
Let $N\ge1$, let $\bm z_1,\ldots,\bm z_N\in\Sph$, and let $C>0$. Suppose that
\begin{equation}
\frac1N\sum_{i=1}^N\|\grad Q(\bm z_i)\|_{\ell_2}
\le C\|Q\|
\qquad\forall Q\in\Pi_t^0.
\label{eq:abstract-gradient-sampling}
\end{equation}
Then for every $H\in\Pi_t^0$,
\begin{equation}
\frac1N\sum_{i=1}^N\|\nabla^2 H(\bm z_i)\|_{\op}
\le\frac{\pi t}{2}C\|H\|.
\label{eq:hessian-sampling-transfer}
\end{equation}
\end{lemma}

\begin{proof}
Lemma~\ref{lem:rotational-estimates} shows that $\Rot_{\bm u}H\in\Pi_{t}^0$. Applying equation~\eqref{eq:hessian-axis-average} at each point $\bm{z}_i$ we obtain
\[
\begin{aligned}
\frac1N\sum_{i=1}^N\|\nabla^2 H(\bm z_i)\|_{\op}
&\le2\int_{\Sph}\frac1N\sum_{i=1}^N
\|\grad\Rot_{\bm u}H(\bm z_i)\|_{\ell_2}\dd\sigma(\bm u)\\
&\overset{(a)}\le2C\int_{\Sph}\|\Rot_{\bm u}H\|\dd\sigma(\bm u)\overset{(b)}\le\frac{\pi t}{2}C\|H\|,
\end{aligned}
\]
where $(a)$ follows from equation~\eqref{eq:abstract-gradient-sampling} by taking $Q=\Rot_{\bm u}H$, and $(b)$ follows from equation~\eqref{eq:rot-average-main}. This proves equation~\eqref{eq:hessian-sampling-transfer}.
\end{proof}

\section{Proof of the main theorem}
\label{sec:tag-criterion}

In this section, we prove Theorem~\ref{thm:coarse-main} by combining the sampling estimates from Section~\ref{sec:sample} with a common regularized gradient flow. We first establish Theorem~\ref{thm:karcher-tail}, which reduces the existence of a spherical design to two inequalities involving $\beta_1$ and $\beta_2$. These inequalities ensure that the gain along the flow exceeds the initial quadrature loss. We then use the geometric estimates from Section~\ref{sec:map} to verify these inequalities for every $q\ge 3t$, thereby completing the proof of Theorem~\ref{thm:coarse-main}.

Let $t\ge1$ and $q\ge1$ be integers.  Set $N=6q^2$ and let $C_1,\ldots,C_N$ be the equal-area cubed-sphere cells constructed in \eqref{xueq81}. For each $i\in\{1,2,\ldots,N\}$,  let $\bm c_i$ be the unique Karcher center of $C_i$.  Let $\mu_m$ and $\beta_j$ be defined as in equation~\eqref{eq:centered-tails}.

The common regularized gradient flow follows Bondarenko, Radchenko, and Viazovska~\cite[Sections~2 and~4]{BRV2013}. With Karcher centers as the initial points, the initial loss is proportional to $\beta_2$, while the gradient and Hessian sampling bounds control the gain along the flow.

\begin{theorem}
\label{thm:karcher-tail}
Let $\beta_1$ and $\beta_2$ be as defined in \eqref{eq:centered-tails}, and set $h(x):=x\ln x+1-x$ for $x>0$. If
\begin{equation}
\beta_1+\beta_2<1
\quad\text{and}\quad
\beta_2<h\bigl(2-\beta_1-\beta_2\bigr),
\label{eq:centered-scalar-criterion}
\end{equation}
then there exists a spherical $t$-design on $\Sph$ consisting of exactly $N$ points.
\end{theorem}

\begin{proof}  
Let 
\begin{equation}\label{xueq47}
s_*=\frac{2\ln  (2-\beta_1-\beta_2)}{\pi t}
\quad\text{and}\quad
\varepsilon=\frac{1}{s_*}\cdot\frac{h(2-\beta_1-\beta_2)-\beta_2}{\pi t(1-\beta_2)}.
\end{equation}
Since $\beta_1+\beta_2<1$ and $\beta_2<h\bigl(2-\beta_1-\beta_2\bigr)$, both $s_*$ and $\varepsilon$ are positive.

For $P\in\Pi_t^0$, define
\[
\bm V_{P}(\bm x)
=\frac{\grad P(\bm x)}
{\sqrt{\|\grad P(\bm x)\|_{\ell_2}^2+\varepsilon^2}}.
\]
For each $i\in\{1,2,\ldots,N\}$, let $\bm{y}_i(P,s)$ be the solution of
\[
\frac{\dd}{\dd s}\bm{y}_i(P,s)
=\bm V_{P}\bigl(\bm{y}_i(P,s)\bigr),
\qquad
\bm{y}_i(P,0)=\bm{c}_i.
\]
Let 
\begin{equation*}
\Omega=\{P\in\Pi_t^0:\|P\|<1\}.
\end{equation*}
The space $\Pi_t^0$ is finite dimensional and $\|\cdot\|$ is a norm on it.  Hence $\Omega$ is a bounded open set in the Hilbert coordinates of Section~\ref{sec:degree-criterion}, and $\partial\Omega$ is compact. Moreover, because $\varepsilon>0$, the smooth vector field $\bm V_{P}$ and its flow depend continuously on $P$. We claim that
\begin{equation}
\frac1N\sum_{i=1}^N
P\bigl(\bm{y}_i(P,s_*)\bigr)>0,
\quad\forall P\in \partial\Omega=\{P\in\Pi_t^0:\|P\|=1\}.
\label{eq:terminal-boundary-positivity}
\end{equation}
Applying Theorem \ref{thm:brouwer-degree-criterion} to
\[
P\longmapsto
\sum_{i=1}^NG_{\bm{y}_i(P,s_*)}
\]
gives $P_0\in\Omega$ such that
\[
\sum_{i=1}^NG_{\bm{y}_i(P_0,s_*)}=0.
\]
By the characterization of spherical designs in Section \ref{sec:degree-criterion}, the points
\[
\bm{y}_1(P_0,s_*),\ldots,\bm{y}_N(P_0,s_*)
\]
form a spherical $t$-design.

It remains to prove equation~\eqref{eq:terminal-boundary-positivity}. Fix a polynomial $P\in \partial\Omega$. Using the chain rule, we have
\begin{equation}\label{xueq13}
\begin{aligned}
\frac{\dd}{\dd s}P(\bm y_i(P,s))
&
=\frac{\|\grad P(\bm y_i(P,s))\|_{\ell_2}^2}
{\sqrt{\|\grad P(\bm y_i(P,s))\|_{\ell_2}^2+\varepsilon^2}}  \ge \|\grad P(\bm y_i(P,s))\|_{\ell_2}-\varepsilon.
\end{aligned}
\end{equation}
By the fundamental theorem of calculus, we have
\begin{equation}
\begin{aligned}
\frac1N\sum_{i=1}^NP(\bm{y}_i(P,s_*))
&\overset{(a)}=\frac1N\sum_{i=1}^NP(\bm{c}_i)+\int_0^{s_*}\frac{\dd}{\dd s} \frac1N\sum_{i=1}^NP(\bm{y}_i(P,s)) \dd s\\
&\overset{(b)}\geq I_1
+I_2  -s_*\cdot \varepsilon,
\end{aligned}
\label{eq:centered-flow-preliminary}
\end{equation}
where
\begin{equation*}
I_1:= \frac1N\sum_{i=1}^NP(\bm c_i)
\quad\text{and}\quad
I_2:=	\int_0^{s_*}  \frac1N\sum_{i=1}^N\|\grad P(\bm{y}_i(P,s))\|_{\ell_2} \dd s.	
\end{equation*}
Here, ($a$) follows from $\bm{y}_i(P,0)=\bm c_i$, and inequality ($b$) follows from equation~\eqref{xueq13}. We claim that
\begin{subequations}
\begin{align}
I_1&
\geq -\frac{2\beta_2}{\pi t(1-\beta_2)},\label{xueq45}\\
I_2&\geq \frac{2h(2-\beta_1-\beta_2)}
{\pi t(1-\beta_2)}\label{xueq46},
\end{align}	
\end{subequations}
where $h(x)=x\ln  x+1-x$. Substituting equation~\eqref{xueq45} and equation~\eqref{xueq46} into equation~\eqref{eq:centered-flow-preliminary}, we obtain
\begin{equation*}
\frac1N\sum_{i=1}^NP(\bm{y}_i(P,s_*))
\geq 
\frac{2h(2-\beta_1-\beta_2)-2\beta_2}{\pi t(1-\beta_2)}-s_*\cdot \varepsilon
=s_*\varepsilon>0.
\end{equation*}
Here, we use equation~\eqref{xueq47}. This proves equation~\eqref{eq:terminal-boundary-positivity}.

It remains to prove equations~\eqref{xueq45} and~\eqref{xueq46}. We first prove equation~\eqref{xueq45}. Since $P\in\Pi_t^0$ and $\|P\|=1$, using  equation~\eqref{eq:centered-value-closure} and equation~\eqref{eq:centered-sampling-constant} we obtain
\begin{equation*}
\begin{aligned}
\left|\frac1N\sum_{i=1}^NP(\bm c_i)\right|
&=\left|\frac1N\sum_{i=1}^N
\left(P(\bm c_i)-N\int_{C_i}P\dd\sigma\right)\right|\\
&\le
\frac1N\sum_{i=1}^N\left|P(\bm c_i)-N\int_{C_i}P\dd\sigma\right|
\le\frac{2\beta_2}{\pi t(1-\beta_2)}.
\end{aligned}
\end{equation*}
Here the equality uses
\[
\sum_{i=1}^N\int_{C_i}P\dd\sigma
=\int_{\Sph}P\dd\sigma=0.
\]
Hence,
\begin{equation}\label{xueq18}
I_1= \frac1N\sum_{i=1}^NP(\bm c_i)\geq -\left|\frac1N\sum_{i=1}^NP(\bm c_i)\right|
\geq -\frac{2\beta_2}{\pi t(1-\beta_2)}.	
\end{equation}
This gives equation~\eqref{xueq45}.

We next prove equation~\eqref{xueq46}. For any $H\in\Pi_t^0$, we define
\begin{equation}\label{xueq23}
\mathcal A_H(s)=\frac1N\sum_{i=1}^N\|\grad H(\bm{y}_i(P,s))\|_{\ell_2}
\quad\text{and}\quad
\mathcal{R}_{\mathrm{max}}(s):=
\sup_{ H\in\Pi_t^0,\|H\|=1}\mathcal A_H(s).
\end{equation}
This supremum is finite and continuous in $s$, because the expression being maximized is jointly continuous in $(H,s)$, and the unit sphere $\{H\in\Pi_t^0:\|H\|=1\}$ is compact. Since $\bm{y}_i(P,0)=\bm c_i$, the quantity $\mathcal{R}_{\mathrm{max}}(0)$ equals $\mathcal{R}_{\mathrm{max}}$ defined in equation~\eqref{xueq12}. By equation~\eqref{eq:centered-sampling-constant} and equation~\eqref{eq:centered-closure}, we have
\begin{equation}
\begin{aligned}
\mathcal{R}_{\mathrm{max}}(0)=\mathcal{R}_{\mathrm{max}}\le\frac1{1-\beta_2}
\end{aligned}
\label{eq:centered-static-input}
\end{equation}
and
\begin{equation}
\begin{aligned}
\mathcal A_H(0)
\ge\bigl(1-\beta_1\mathcal{R}_{\mathrm{max}}\bigr)\|H\|
\ge\frac{1-\beta_1-\beta_2}{1-\beta_2}\|H\|,
\quad \forall H\in\Pi_t^0.
\end{aligned}
\label{eq:centered-static-input2}
\end{equation}
For fixed $i$, since the vector field $\bm V_P$ is smooth, the trajectory $r\mapsto \bm{y}_i(P,r)$ is a smooth curve. Lemma~\ref{lem:spherical-gradient-variation-curve} gives
\begin{equation}\label{xueq19}
\begin{aligned}
\bigg|\|\grad H(\bm{y}_i(P,s))\|_{\ell_2}
-\|\grad H(\bm{y}_i(P,0))\|_{\ell_2}\bigg|
\le&
\int_0^s\|\nabla^2 H(\bm{y}_i(P,r))\|_{\op}
\bigg\|\frac{\dd}{\dd r}\bm{y}_i(P,r)\bigg\|_{\ell_2}\dd r \\
\le&
\int_0^s\|\nabla^2 H(\bm{y}_i(P,r))\|_{\op}\dd r,
\end{aligned}
\end{equation}
where the last inequality follows from $\|\frac{\dd}{\dd r}\bm{y}_i(P,r)\|_{\ell_2}\le1$.   By equation~\eqref{eq:hessian-sampling-transfer} and the definition of $\mathcal{R}_{\mathrm{max}}(r)$,  we have
\begin{equation}\label{xueq20}
\frac1N\sum_{i=1}^N\|\nabla^2 H(\bm{y}_i(P,r))\|_{\op}
\le\frac{\pi t}{2}\mathcal{R}_{\mathrm{max}}(r)\|H\|.
\end{equation}
Hence, we have
\begin{equation} 
\begin{aligned}
|\mathcal A_H(s)-\mathcal A_H(0)|
&\leq \frac1N\sum_{i=1}^N\bigg|\|\grad H(\bm{y}_i(P,s))\|_{\ell_2}-\|\grad H(\bm{y}_i(P,0))\|_{\ell_2} \bigg|\\
&\overset{(a)}\leq \frac1N\sum_{i=1}^N \int_0^s\|\nabla^2 H(\bm{y}_i(P,r))\|_{\op}\dd r \overset{(b)}\le\frac{\pi t}{2}\int_0^s\mathcal{R}_{\mathrm{max}}(r)\dd r\,\|H\|,
\label{eq:centered-dynamic-integral}	
\end{aligned}
\end{equation}
where ($a$) follows from equation~\eqref{xueq19} and ($b$) follows from equation~\eqref{xueq20}. Then we have
\begin{equation*}
\begin{aligned}
\mathcal A_H(s)
&\le\mathcal A_H(0)+\frac{\pi t}{2}\int_0^s\mathcal{R}_{\mathrm{max}}(r)\dd r\,\|H\|	\\
&\overset{(a)}\leq \mathcal{R}_{\mathrm{max}}(0)\|H\|+\frac{\pi t}{2}\int_0^s\mathcal{R}_{\mathrm{max}}(r)\dd r\,\|H\|\\
&\overset{(b)}\leq \frac1{1-\beta_2}\|H\|+\frac{\pi t}{2}\int_0^s\mathcal{R}_{\mathrm{max}}(r)\dd r\,\|H\|.
\end{aligned}
\end{equation*}
Here, ($a$) follows from the definition of $\mathcal{R}_{\mathrm{max}}(0)$, and ($b$) follows from equation~\eqref{eq:centered-static-input}. Taking the supremum over $H\in\Pi_t^0$ with $\|H\|=1$, we obtain
\begin{equation*}
\mathcal{R}_{\mathrm{max}}(s)\le 	\frac1{1-\beta_2} +\frac{\pi t}{2}\int_0^s\mathcal{R}_{\mathrm{max}}(r)\dd r.	
\end{equation*}
Gronwall's inequality shows that if $a\ge 0$, $b\ge 0$, and  $f:[0,u]\to[0,\infty)$ is a continuous function satisfying
\[
f(s)\le a+b\int_0^s f(r)\,dr,\quad \forall s\in[0,u],
\]
then $f(s)\leq a e^{bs}$ for every $s\in[0,u]$ \cite{Bellman1953,Gronwall1919}. Hence, applying Gronwall's inequality we obtain
\[
\mathcal{R}_{\mathrm{max}}(s)\le\frac{e^{\pi ts/2}}{1-\beta_2}.
\]
Then for any $H\in\Pi_t^0$,
\begin{equation}
\begin{aligned}
\mathcal A_H(s)
&\overset{(a)}\ge \mathcal A_H(0)-\frac{\pi t}{2}\int_0^s\mathcal{R}_{\mathrm{max}}(r)\dd r\,\|H\|\\
&\overset{(b)}\ge  \frac{1-\beta_1-\beta_2}{1-\beta_2}\|H\|-\frac{\pi t}{2}\int_0^s\frac{e^{\pi tr/2}}{1-\beta_2}\dd r\,\|H\|\\
&= \frac{2-\beta_1-\beta_2-e^{\pi ts/2}}{1-\beta_2}\|H\|,	
\end{aligned}
\label{eq:centered-dynamic-lower}
\end{equation}
where ($a$) follows from equation~\eqref{eq:centered-dynamic-integral} and ($b$) follows from equation~\eqref{eq:centered-static-input2}. Taking $H=P$ in equation~\eqref{eq:centered-dynamic-lower} and integrating gives
\begin{equation}\label{xueq21}
I_2=\int_0^{s_*}  \mathcal A_P(s) \dd s\geq \int_0^{s_*}  \frac{2-\beta_1-\beta_2-e^{\pi ts/2}}{1-\beta_2} \dd s
=\frac{2h(2-\beta_1-\beta_2)}
{\pi t(1-\beta_2)}.
\end{equation}
We arrive at equation~\eqref{xueq46}. This completes the proof.

\end{proof}

Now we give a proof of Theorem~\ref{thm:coarse-main}.

\begin{proof}[Proof of Theorem~\ref{thm:coarse-main}]
If $t=1$, take $N/2=3q^2$ pairwise disjoint antipodal pairs on $\Sph$. Their vector sum is zero.  Since every polynomial of degree at most one is a constant plus a linear function, these points form a spherical $1$-design.

Now let $t\ge2$, let $q\ge3t$, and set $N=6q^2$. By Remark~\ref{rem:cell-karcher-center}, the cells have unique Karcher centers satisfying the first-variation identity equation~\eqref{eq:karcher-first-variation}.  Moreover, Lemma~\ref{lem:karcher-localization} gives
\[
\begin{aligned}
C_i
\subset B\left(\bm c_i,\frac{15}{\sqrt{8N}}\right)
\end{aligned}
\]
and
\[
\int_{C_i}\dist(\bm c_i,\bm y)^2\dd\sigma(\bm y)
<\frac4{9q^2N}.
\]
Since $\sigma(C_i)=1/N$, the Cauchy-Schwarz inequality and the second-moment estimate give
\[
\begin{aligned}
\int_{C_i}\dist(\bm c_i,\bm y)\dd\sigma(\bm y)
&\le
\left(
\int_{C_i}\dist(\bm c_i,\bm y)^2\dd\sigma(\bm y)
\right)^{1/2}
\left(\int_{C_i}1\dd\sigma(\bm y)\right)^{1/2}<\frac2{3qN}.
\end{aligned}
\]
Moreover, the support estimate implies, for every integer $m\ge2$,
\[
\begin{aligned}
\int_{C_i}\dist(\bm c_i,\bm y)^m\dd\sigma(\bm y)
&\le
\left(\frac{5\sqrt3}{4q}\right)^{m-2}
\int_{C_i}\dist(\bm c_i,\bm y)^2\dd\sigma(\bm y)<
\frac4{9q^2N}
\left(\frac{5\sqrt3}{4q}\right)^{m-2}.
\end{aligned}
\]
These estimates hold for every $1\leq i\leq N$.   Hence, by the definition of $\mu_j$ in \eqref{eq:centered-tails}, we have
\[
\mu_1<\frac2{3q},
\qquad
\mu_m<\frac4{9q^2}
\left(\frac{5\sqrt3}{4q}\right)^{m-2},
\quad\forall m\ge2.
\]
Since $\frac{t}{q}\leq \frac{1}{3}$, the definitions of the moment tails therefore give
\begin{equation*}
\begin{aligned}
\beta_2&<\frac4{9q^2}\left(\frac{4q}{5\sqrt3}\right)^2
\sum_{m=2}^{\infty}\frac1{m!}
\left(\frac{5\sqrt3\pi t}{8q}\right)^m
\leq \frac4{9q^2}\left(\frac{4q}{5\sqrt3}\right)^2
\sum_{m=2}^{\infty}\frac1{m!}
\left(\frac{5\sqrt3\pi }{24}\right)^m\\
&= \frac{64}{675}
\left(e^{\frac{5\sqrt3\pi}{24}}-1-\frac{5\sqrt3\pi}{24}\right)<\frac{13}{140}.
\end{aligned}
\end{equation*}
Since $\beta_1=\frac{\pi t}{2}\mu_1+\beta_2$, $\mu_1<\frac{2}{3q}$ and $\frac{t}{q}\leq \frac{1}{3}$,  we further have
\[
\beta_1+\beta_2=\frac{\pi t}{2}\mu_1+2\beta_2<\frac{\pi t}{3q}+2\beta_2\leq\frac\pi9+2\beta_2<\frac\pi9+\frac{13}{70}<\frac7{13}<1.
\]
The function $h(x)=x\ln x+1-x$ is increasing for $x>1$, and $2-\beta_1-\beta_2>\frac{19}{13}$.   Therefore,
\[
\begin{aligned}
h(2-\beta_1-\beta_2)
&>h\left(\frac{19}{13}\right)>\frac{13}{140}>\beta_2.
\end{aligned}
\]
Applying Theorem~\ref{thm:karcher-tail} gives a spherical $t$-design consisting of exactly $N=6q^2$ points.

In particular, taking $q=3t$, we obtain a spherical $t$-design in $\Sph$ of size $N=6q^2=54t^2$.
\end{proof}

\section*{Use of Artificial Intelligence}
During the preparation of this manuscript, the authors used ChatGPT  for assistance with filling in technical details,  and improving the language. All mathematical ideas, proof strategies, and proofs presented in this manuscript are human-generated and were developed and verified by the authors.  The authors take full responsibility for the correctness and originality of the results.

\appendix

\section{Proof of Lemma~\ref{lem:rotational-estimates}}
\label{pf-lem:rotational-estimates}

This appendix proves the rotational identities and estimates used in Section~\ref{sec:sample}.  Rotation invariance gives the scalar averaging identities, the one-dimensional Bernstein inequality controls their integrals, and a coordinate calculation transfers the gradient bounds to the spherical Hessian.

\begin{proof}[Proof of Lemma~\ref{lem:rotational-estimates}]

\par\smallskip
\noindent \textup{(i)} 
The coordinate formula in equation~\eqref{eq:rotational-derivative-coordinate} proves the degree part directly. Indeed, if $\nabla_{\R^3}P(\bm x)=(g_1,g_2,g_3)^T$, then $g_1,g_2,g_3\in\Pi_{t-1}$, and every coefficient $u_2x_3-u_3x_2$, $u_3x_1-u_1x_3$, $u_1x_2-u_2x_1$ is linear in $\bm x$. Hence $\Rot_{\bm u}P\in\Pi_t$.

The zero-mean conclusion uses the definition equation~\eqref{eq:rotational-derivative}.  Since $\sigma$ is invariant under rotations, and since differentiation under the integral is justified by smoothness on the compact sphere,
\[
\int_{\Sph}\Rot_{\bm u}P(\bm x)\dd\sigma(\bm x)
=\left.\frac{\dd}{\dd s}\right|_{s=0}
\int_{\Sph}P(\mathcal R_{\bm u}(s)\bm x)\dd\sigma(\bm x)
=\left.\frac{\dd}{\dd s}\right|_{s=0}
\int_{\Sph}P(\bm x)\dd\sigma(\bm x)
=0.
\]
Therefore $\Rot_{\bm u}P\in\Pi_t^0$.

\par\smallskip
\noindent \textup{(ii)}
Since $\sigma$ is the normalized measure on $\Sph$, a direct calculation gives
\begin{equation}
\int_{\Sph}|\langle\bm u,\bm z\rangle|\dd\sigma(\bm u)
=\frac12\|\bm z\|_{\ell_2}
\qquad\forall \bm z\in\R^3.
\label{eq:three-dimensional-mean-projection}
\end{equation}
Applying equation~\eqref{eq:three-dimensional-mean-projection} with $\bm z=\bm x\times\grad P(\bm x)$ gives
\begin{equation}\label{eq:three-dimensional-mean-projection2}
\int_{\Sph}|\langle\bm u,\bm x\times\grad P(\bm x)\rangle|\dd\sigma(\bm u)
=\frac12\|\bm x\times\grad P(\bm x) \|_{\ell_2}.
\end{equation}
Since $\grad P(\bm x)\in T_{\bm x}\Sph$,  by equation~\eqref{eq:rotational-derivative-coordinate} we have $\|\bm x\times\grad P(\bm x)\|_{\ell_2} =\|\grad P(\bm x)\|_{\ell_2}$ and
\[
\Rot_{\bm u}P(\bm x)
=\left\langle\grad P(\bm x),\bm u\times\bm x\right\rangle
=\left\langle\bm u,\bm x\times\grad P(\bm x)\right\rangle.
\]
Then, equation~\eqref{eq:three-dimensional-mean-projection2} becomes
\[
\int_{\Sph}|\Rot_{\bm u}P(\bm x)|\dd\sigma(\bm u)
=\frac12\|\grad P(\bm x)\|_{\ell_2},
\]
which is equation~\eqref{eq:axis-average-main}.

\par\smallskip
\noindent \textup{(iii)}
We use Zygmund's integral Bernstein inequality, recorded in Arestov~\cite[equation~(1.2), p.~1]{Arestov1982}.  If $f$ is a trigonometric polynomial of degree at most $t$, then that inequality with $\varphi(u)=u$ and derivative order one gives
\[
\int_0^{2\pi}|f'(\theta)|\dd\theta
\le t\int_0^{2\pi}|f(\theta)|\dd\theta.
\]

Let $\eta$ be normalized Haar measure on $\mathrm{SO}(3)$ (see~\cite[Chapter~2]{Folland2016}).  Let
\[
\bm a_\theta=(\cos\theta,\sin\theta,0),
\qquad
\bm b_\theta=(-\sin\theta,\cos\theta,0).
\]
For $\bm G\in\mathrm{SO}(3)$, the curve $\gamma_{\bm G}(\theta)=\bm G\bm a_\theta$ is a unit-speed great circle with unit tangent vector $\bm G\bm b_\theta$.  If $Q\in\Pi_t$, then $f_{\bm G}(\theta)=Q(\bm G\bm a_\theta)$ is a trigonometric polynomial of degree at most $t$, and
\[
f_{\bm G}'(\theta)
=\left\langle\grad Q(\bm G\bm a_\theta),\bm G\bm b_\theta\right\rangle.
\]
Applying the preceding one-dimensional Bernstein inequality to $f_{\bm G}$ gives
\begin{equation}
\int_0^{2\pi}
\left|\left\langle
\grad Q(\bm G\bm a_\theta),\bm G\bm b_\theta
\right\rangle\right|\dd\theta
\le t\int_0^{2\pi}|Q(\bm G\bm a_\theta)|\dd\theta.
\label{eq:great-circle-Bernstein}
\end{equation}
Averaging equation~\eqref{eq:great-circle-Bernstein} over $\mathrm{SO}(3)$ therefore yields
\begin{equation}
\label{xueq29}
\begin{aligned}
\int_{\mathrm{SO}(3)}\int_0^{2\pi}
|\langle \grad Q(\bm G\bm a_\theta),\bm G\bm b_\theta\rangle|
\dd\theta\dd\eta(\bm G)
\leq t \int_{\mathrm{SO}(3)}\int_0^{2\pi}
|Q(\bm G\bm a_\theta)|\dd\theta\dd\eta(\bm G).	
\end{aligned}	
\end{equation}
For each fixed $\theta$, the push-forward of $\eta$ under
\[
\bm G\longmapsto(\bm G\bm a_\theta,\bm G\bm b_\theta)
\]
is the rotation-invariant probability measure on the unit tangent bundle of $\Sph$.  Thus the first component has distribution $\sigma$ and, conditional on a point $\bm x$, the second component is uniformly distributed on
\[
S^1_{\bm x}=\{\bm v\in T_{\bm x}\Sph:\|\bm v\|_{\ell_2}=1\}.
\]
If $\nu_{\bm x}$ denotes normalized angular measure on this unit circle, then, for $\bm w\in T_{\bm x}\Sph$,
\begin{equation}
\int_{S^1_{\bm x}}|\langle\bm w,\bm v\rangle|
\dd\nu_{\bm x}(\bm v)
=\frac{\|\bm w\|_{\ell_2}}{2\pi}\int_0^{2\pi}|\cos\varphi|\dd\varphi
=\frac2\pi\|\bm w\|_{\ell_2}.
\label{eq:tangent-mean-projection}
\end{equation}
Hence, for fixed $\theta\in[0,2\pi]$, we have
\begin{equation*}
\begin{aligned}
\int_{\mathrm{SO}(3)}
|\langle \grad Q(\bm G\bm a_\theta),\bm G\bm b_\theta\rangle|
\dd\eta(\bm G)
&=\int_{\Sph}\int_{S^1_{\bm x}}
|\langle \grad Q(\bm x), \bm v\rangle|
\dd\nu_{\bm x}(\bm v)\dd\sigma(\bm x)\\
&=
\frac2\pi
\int_{\mathbb S^2}\|\grad Q(\bm x)\|_{\ell_2}\dd\sigma(\bm x),
\end{aligned}
\end{equation*}
where the last equality follows from equation~\eqref{eq:tangent-mean-projection}. Then Fubini's theorem gives
\begin{equation}
\label{xueq28}
\int_{\mathrm{SO}(3)}\int_0^{2\pi}
|\langle \grad Q(\bm G\bm a_\theta),\bm G\bm b_\theta\rangle|
\dd\theta\dd\eta(\bm G)
=\int_{0}^{2\pi}
\frac2\pi
\int_{\mathbb S^2}\|\grad Q\|_{\ell_2}\dd\sigma\dd\theta
=
4\int_{\mathbb S^2}\|\grad Q\|_{\ell_2}\dd\sigma.
\end{equation}
Similarly, we have
\begin{equation}
\label{xueq30}
\int_{\mathrm{SO}(3)}\int_0^{2\pi}
|Q(\bm G\bm a_\theta)|\dd\theta\dd\eta(\bm G)
=\int_{0}^{2\pi} \int_{\Sph}|Q(\bm x)|\dd\sigma(\bm x)\dd\theta
=2\pi\int_{\Sph}|Q(\bm x)|\dd\sigma(\bm x).
\end{equation}
Substituting equations~\eqref{xueq28} and~\eqref{xueq30} into equation~\eqref{xueq29}, we arrive at
\[
\int_{\Sph}\|\grad Q(\bm x)\|_{\ell_2}\dd\sigma(\bm x)
\le\frac{\pi t}{2}\int_{\Sph}|Q(\bm x)|\dd\sigma(\bm x),
\]
which proves equation~\eqref{eq:spherical-L1-Bernstein}. 

\par\smallskip
\noindent \textup{(iv)} By  \textup{(i)}, $\Rot_{\bm u}P\in\Pi_t$.  Hence
\textup{(iii)} gives, for each fixed $\bm u$,
\[
\|\Rot_{\bm u}P\|
=\int_{\Sph}\|\grad\Rot_{\bm u}P(\bm x)\|_{\ell_2}\dd\sigma(\bm x)
\le\frac{\pi t}{2}
\int_{\Sph}|\Rot_{\bm u}P(\bm x)|\dd\sigma(\bm x).
\]
Integrating the preceding inequality with respect to $\bm u$, we obtain
\[
\begin{aligned}
\int_{\Sph}\|\Rot_{\bm u}P\|\dd\sigma(\bm u)
&\le\frac{\pi t}{2}\int_{\Sph}\int_{\Sph}
|\Rot_{\bm u}P(\bm x)|\dd\sigma(\bm x)\dd\sigma(\bm u)=\frac{\pi t}{4}\|P\|,
\end{aligned}
\]
where the last equality follows from equation~\eqref{eq:axis-average-main}. We arrive at equation~\eqref{eq:rot-average-main}.

\par\smallskip
\noindent \textup{(v)} By rotation invariance, it is
enough to prove the estimate at $\bm e_3=(0,0,1)$.  Indeed, if $\bm R\in\mathrm{SO}(3)$ and $\widetilde P(\bm y)=P(\bm R\bm y)$, then Hessian norms are preserved under the rotation, while $\Rot_{\bm u}\widetilde P$ is transformed into $\Rot_{\bm R\bm u}P$. The measure $d\sigma(\bm u)$ is also rotation invariant.

Write
\[
p_j=\frac{\partial P}{\partial x_j}(\bm e_3),
\qquad
p_{ij}=\frac{\partial^2 P}{\partial x_i\partial x_j}(\bm e_3).
\]
In the tangent basis $\bm e_1=(1,0,0),\bm e_2=(0,1,0)$ of $T_{\bm e_3}\Sph$, the spherical Hessian has matrix
\[
\bm B=
\begin{pmatrix}
p_{11}-p_3 & p_{12}\\
p_{12} & p_{22}-p_3
\end{pmatrix}.
\]
This is just the definition of the spherical Hessian, because the tangential projection at $\bm e_3$ keeps only the first two coordinates and $\langle\nabla_{\R^3}P(\bm e_3),\bm e_3\rangle=p_3$.

If $\bm B=0$, the desired estimate is immediate.  Otherwise choose a unit eigenvector $\bm v=a\bm e_1+b\bm e_2\in T_{\bm e_3}\Sph$ such that
\[
\bm B\binom{a}{b}
=\lambda\binom{a}{b},
\qquad
a^2+b^2=1,
\qquad
|\lambda|=\|\bm B\|_{\op}
=\|\nabla^2P(\bm e_3)\|_{\op}.
\]
We now compute the $\bm v$-directional derivative of $\Rot_{\bm u}P$ directly from the coordinate formula equation~\eqref{eq:rotational-derivative-coordinate}.  Let $\gamma$ be any smooth curve on $\Sph$ with $\gamma(0)=\bm e_3$ and $\gamma'(0)=\bm v$.  Since $\gamma_1'(0)=a$, $\gamma_2'(0)=b$, and $\gamma_3'(0)=0$, writing $P_j(\bm y)=\partial P(\bm y)/\partial y_j$ and differentiating
\[
\begin{aligned}
\Rot_{\bm u}P(\bm y)
&=(u_2y_3-u_3y_2)P_1(\bm y)
+(u_3y_1-u_1y_3)P_2(\bm y)  +(u_1y_2-u_2y_1)P_3(\bm y)
\end{aligned}
\]
at $\bm y=\bm e_3$ in the direction $\bm v$ gives
\[
\begin{aligned}
\left\langle\grad\Rot_{\bm u}P(\bm e_3),\bm v\right\rangle
&=
u_1(-p_{12}a-(p_{22}-p_3)b)+
u_2((p_{11}-p_3)a+p_{12}b)
+u_3(ap_2-bp_1).
\end{aligned}
\]
Using the eigenvector equations
\[
(p_{11}-p_3)a+p_{12}b=\lambda a,
\qquad
p_{12}a+(p_{22}-p_3)b=\lambda b,
\]
this becomes
\[
\left\langle\grad\Rot_{\bm u}P(\bm e_3),\bm v\right\rangle
=-\lambda b\,u_1+\lambda a\,u_2+(ap_2-bp_1)u_3.
\]
Define
\[
\bm z=(-\lambda b,\lambda a,ap_2-bp_1)\in\R^3.
\]
Then
\[
\left\langle\grad\Rot_{\bm u}P(\bm e_3),\bm v\right\rangle
=\langle\bm u,\bm z\rangle.
\]
Therefore, for every $\bm u\in\Sph$,
\[
\|\grad\Rot_{\bm u}P(\bm e_3)\|_{\ell_2}
\ge
\left|
\left\langle\grad\Rot_{\bm u}P(\bm e_3),\bm v\right\rangle
\right|
=|\langle\bm u,\bm z\rangle|.
\]
Moreover,
\[
\|\bm z\|_{\ell_2}^2
=\lambda^2b^2+\lambda^2a^2+(ap_2-bp_1)^2
=\lambda^2+(ap_2-bp_1)^2
\ge\lambda^2.
\]
Integrating in $\bm u$ and applying equation~\eqref{eq:three-dimensional-mean-projection} yields
\[
\begin{aligned}
\int_{\Sph}
\|\grad\Rot_{\bm u}P(\bm e_3)\|_{\ell_2}\dd\sigma(\bm u)
&\ge
\int_{\Sph}|\langle\bm u,\bm z\rangle|\dd\sigma(\bm u)=\frac12\|\bm z\|_{\ell_2}
\ge\frac12\|\nabla^2P(\bm e_3)\|_{\op}.
\end{aligned}
\]
Multiplying by $2$ proves equation~\eqref{eq:hessian-axis-average} at $\bm e_3$, and the rotation-invariance reduction proves it at every $\bm x\in\Sph$.

\end{proof}

\section{Proof of Lemma~\ref{lem:karcher-localization}}
\label{app:karcher-localization-proof}

This appendix proves the radius and second-moment bounds in Lemma~\ref{lem:karcher-localization}.  We first bound the metric of the cubed-sphere map and the displacement of a Karcher center.  A weighted variance inequality on the parameter square, followed by a chord-arc comparison, then yields the required second distance moment. The metric information needed below is recorded in the next lemma. At every point $(s,r)$ with $|s|\ne|r|$, let
\[
\bm H:=\bm H(s,r)=\bm J_F(s,r)^T\bm J_F(s,r)=(H_{ij})_{i,j=1}^2,
\]
where $\bm J_F:=\bm J_F(s,r)\in\mathbb{R}^{3\times 2}$ is the Jacobian matrix of the map $F$ defined in equation~\eqref{eq:direct-face-map} at $(s,r)$. The Jacobian and the metric matrix are understood almost everywhere since the formula seams $\{(s,r)\in[-1,1]^2:|s|=|r|\}$ have planar measure zero.

\begin{lemma}
\label{prop:cubed-sphere-metric}
At every smooth interior point $(s,r)\in(-1,1)^2$ with $|s|\ne|r|$, we have
\begin{equation}
\tr\bm H<\frac76
\quad\text{and}\quad
\lambda_{\max}(\bm H)<\frac{27}{32}.
\label{eq:metric-trace-bound}
\end{equation}
Moreover
\begin{equation}
\dist\bigl(F(\bm p),F(\bm q)\bigr)
< \frac{3\sqrt{6}}{8}\,\|\bm p-\bm q\|_{\ell_2}
\qquad\forall \bm p,\bm q\in[-1,1]^2,\ \bm p\ne\bm q.
\label{eq:metric-global-lipschitz}
\end{equation}
\end{lemma}

\begin{proof}
By reflection and coordinate-interchange symmetry, it suffices to work in $0\le r<s<1$.  Let
\[
\theta=\frac{\pi r}{12s}\in[0,\frac{\pi}{12}),\qquad
g(\theta)=\sqrt2-\cos\theta.
\]
In particular, $g(\theta)\ge\sqrt2-1$. Write
\[
\bm T=\bigl(T_1(s,r),T_2(s,r)\bigr)^T
\quad\text{and}\quad
w=1-\frac{\|\bm T\|_{\ell_2}^2}{4}=1-\frac{s^2}{4}\left(4-\frac{\sqrt2}{g(\theta)}\right).
\]
Note that  $\frac{\sqrt2}{4g(\theta)}<w<1$ and
\begin{equation*}
F(s,r)=\left(\sqrt{w}\cdot T_1(s,r),\sqrt{w}\cdot T_2(s,r), 1-\frac{T_1(s,r)^2+T_2(s,r)^2}{2}\right).
\end{equation*}
Let $\bm J_T=(\partial_s\bm T,\partial_r\bm T)\in\mathbb{R}^{2\times 2}$ be the Jacobian matrix of the map $\bm T$.  Differentiating the formulas for $\bm T$ and $F$ gives
\[
\qquad
\bm H=w\bm J_T^T\bm J_T+
\frac{1+w}{4w}\bm J_T^T\bm T\bm T^T\bm J_T.
\]
Substitution yields
\begin{equation}
\begin{aligned}
\tr\bm H
&=f(w,\theta):=
\frac{A(\theta)}{w}+
B(\theta)w,
\qquad
\det\bm H=\frac{\pi^2}{36}.
\end{aligned}
\label{eq:metric-direct-formula}
\end{equation}
Here,
\begin{equation*}
\begin{aligned}
A(\theta)&:=\sqrt2\cdot \frac{
(2g(\theta)(2\sqrt2g(\theta)-1)-\theta\sin\theta)^2
+\frac{\pi^2}{144}\sin^2\theta}
{4g(\theta)^3(2\sqrt2g(\theta)-1)}>0,\\
B(\theta)&:=
\frac{2\sqrt2g(\theta)(\theta^2+\frac{\pi^2}{144})}{2\sqrt2g(\theta)-1}.	
\end{aligned}	
\end{equation*}
For fixed $\theta$, the trace $\tr\bm H= \frac{A(\theta)}{w}+ B(\theta)w$ is convex in $w$. Note that $\frac{\sqrt2}{4g(\theta)}<w<1$. Hence, for each $\theta\in[0,\frac{\pi}{12})$, we have $\tr\bm H\leq \max\{ f_1(\theta), f_2(\theta)\}$, where
\begin{align}
f_1(\theta)
&:=f\Big(\frac{\sqrt2}{4g(\theta)},\theta\Big)
=4-\frac{\sqrt2}{g(\theta)}
-\frac{\sqrt2\theta\sin\theta}{g(\theta)^2}
+(\theta^2+\frac{\pi^2}{144})
\left(1+\frac{\sqrt2}{4g(\theta)}+\frac{\sqrt2}{4g(\theta)^3}\right),
\label{eq:metric-tau-zero}\\
f_2(\theta)
&:=f\Big(1,\theta\Big) 
=4(3-2\sqrt2\cos\theta)
-\frac{4\theta\sin\theta}{g(\theta)}
+\frac{\theta^2+\frac{\pi^2}{144}}{g(\theta)^2}.
\label{eq:metric-tau-one}
\end{align}
Both functions are strictly increasing on $[0,\frac{\pi}{12}]$.  Indeed,
\[
\frac{\dd f_1}{\dd\theta}
=
\frac{\sqrt2(g(\theta)^2+3)}{4g(\theta)^4}
\left[2\theta g(\theta)(2\sqrt2g(\theta)-1)-(\theta^2+\frac{\pi^2}{144})\sin\theta\right].
\]
For $0<\theta\le \frac{\pi}{12}$, its bracket is at least $2\theta(5\sqrt2-7-\frac{\pi^2}{144})>0$, because
\[
g(\theta)(2\sqrt2g(\theta)-1)\ge(\sqrt2-1)^3=5\sqrt2-7
\quad\text{and}\quad
(\theta^2+\frac{\pi^2}{144})\sin\theta\leq 2\cdot \frac{\pi^2}{144}\cdot \theta.
\]
Hence, $\frac{\dd f_1}{\dd\theta}>0$ for $0<\theta\le \frac{\pi}{12}$.

For the other endpoint, differentiation gives
\[
\begin{aligned}
\frac{\dd f_2}{\dd\theta}
=\frac2{g(\theta)^3}(\sin\theta\cdot \,I_1+I_2 ).
\end{aligned}
\]
where
\begin{equation*}
I_1:=g(\theta)(1+2g(\theta))(2\sqrt2g(\theta)-1)-\theta^2-\frac{\pi^2}{144}
\quad\text{and}\quad
I_2:=(\theta-\sin\theta)g(\theta)(2\sqrt2g(\theta)-1)	.
\end{equation*}
The term $I_2$ is nonnegative.  To bound $I_1$, we use
\[
g(\theta)(1+2g(\theta))\ge5-3\sqrt2>\frac34,
\qquad
2\sqrt2g(\theta)-1
>\frac16+\frac{14}{5}
\left(\frac{\theta^2}{2}-\frac{\theta^4}{24}\right).
\]
The latter inequality follows from $1-\cos\theta\ge\theta^2/2-\theta^4/24$.  Hence, since $\theta\leq \frac{\pi}{12}<\frac13$,
\[
I_1>
\frac18-\frac{\pi^2}{144}+
\theta^2\left(\frac1{20}-\frac{7\theta^2}{80}\right)>0.
\]
Thus $\dd f_2/\dd\theta>0$ for $0<\theta\le \frac{\pi}{12}$.

Since both $ f_1(\theta)$ and $ f_2(\theta)$ are increasing when $0<\theta\le \frac{\pi}{12}$, we have $\tr\bm H\leq \max\{ f_1(\frac{\pi}{12}), f_2(\frac{\pi}{12})\}$. Since
\[
\begin{aligned}
f_1(\frac{\pi}{12})
&=2-\frac{2\sqrt3}{3}
-\frac{\pi(1+\sqrt3)}{18}
+\frac{\pi^2(63+23\sqrt3)}{1296}<\frac76,\\
f_2(\frac{\pi}{12})
&=8-4\sqrt3-\frac{\pi\sqrt3}{9}
+\pi^2\left(\frac1{27}+\frac{\sqrt3}{54}\right)<\frac76,
\end{aligned}
\]
we obtain $\tr\bm H< \frac76$. Note that
\[
(\tr\bm H)^2-4\det\bm H
<
\left(\frac76\right)^2-\frac19\left(\frac{157}{50}\right)^2
=\frac{166}{625}
<\left(\frac{13}{25}\right)^2.
\]
Consequently, since $\bm H$ is a $2\times 2$ matrix, we have
\[
\lambda_{\max}(\bm H)
=\frac{\tr\bm H+\sqrt{(\tr\bm H)^2-4\det\bm H}}2
<\frac12\left(\frac76+\frac{13}{25}\right)
=\frac{253}{300}<\frac{27}{32}.
\]
This proves equation~\eqref{eq:metric-trace-bound}.

It remains to prove equation~\eqref{eq:metric-global-lipschitz}. The map $\bm T$, and hence $F$, is continuous on the closed parameter square.  First let $\bm p,\bm q$ be interior points whose connecting segment is contained in neither formula seam.  Split the segment at its finitely many seam intersections.  On each open smooth subsegment, the bound $\|\bm J_F\|_{\op}\le\sqrt{253/300}$ shows that the length of its image is at most $\sqrt{253/300}$ times its Euclidean length. Apply this estimate first on closed subsegments within the smooth sector and then pass to their endpoints by continuity.  Since geodesic distance is bounded by curve length, the triangle inequality gives $\dist(F(\bm p),F(\bm q))\le\sqrt{253/300}\|\bm p-\bm q\|_{\ell_2}$. For arbitrary endpoints in the closed square, approximate them by interior endpoints whose connecting segment is contained in neither seam.  The bound just proved applies to each approximating segment. Continuity then gives
\[
\dist\bigl(F(\bm p),F(\bm q)\bigr)
\le \sqrt{\frac{253}{300}}\,\|\bm p-\bm q\|_{\ell_2}
<\frac{3\sqrt{6}}{8}\,\|\bm p-\bm q\|_{\ell_2}
\]
whenever $\bm p\ne\bm q$.  This proves equation~\eqref{eq:metric-global-lipschitz}.   Rotations preserve spherical distance and the metric matrix, so the same estimates hold for all six face maps.

\end{proof}

\begin{lemma}
\label{app-lemma1}
Let $C$ be a closed cubed-sphere cell defined in \eqref{xueq81}, which satisfies  $\sigma(C)=1/N$. Let $\bm c$ be the unique Karcher center of $C$. Suppose that for some $\bm x\in\Sph$ and $0<R\leq \frac1{3\sqrt{2}}$, we have
\begin{equation}\label{xueq64}
C\subset B(\bm x,R),
\qquad
N\int_{C}\dist(\bm x,\bm y)^2\dd\sigma(\bm y)<\frac{R^2}{3}.
\end{equation}
Then
\begin{equation}\label{xueq50}
C
\subset B\left(\bm c,\frac{5R}{3}\right).
\end{equation}
\end{lemma}

\begin{proof}
All the hypotheses and the conclusion are invariant under rotations of $\Sph$.  We may therefore assume that
\begin{equation*}
\bm c=\bm e_3=(0,0,1).
\end{equation*}
Since $\bm c$ is the minimizer of the Karcher energy, the triangle inequality and Minkowski's inequality give
\begin{equation}\label{xueq49}
\begin{aligned}
\rho:=\dist(\bm c,\bm x)
&\le
\left(N\int_{C}\dist(\bm c,\bm y)^2\dd\sigma(\bm y)\right)^{1/2}
+\left(N\int_{C}\dist(\bm x,\bm y)^2\dd\sigma(\bm y)\right)^{1/2}<\frac{2R}{\sqrt3}.
\end{aligned}
\end{equation}
If $\rho=0$, equation~\eqref{xueq50} follows immediately from the containment hypothesis.  We therefore assume $\rho>0$. We claim that
\begin{equation}\label{xueq48}
\sin\rho
<
\frac{3R}{\sin(3R)}\frac{R}{\sqrt3}.	
\end{equation}
Since $g(s):=s/\sin s$ is increasing, we have
\begin{equation}
\begin{aligned}
\rho=\frac{\rho}{\sin\rho}\sin\rho 
\overset{(a)}<g\Big (\frac{2R}{\sqrt3}\Big)
\sin\rho
\overset{(b)}<g\Big (\frac{2R}{\sqrt3}\Big)g\Big (3R\Big)
\frac{R}{\sqrt3},
\end{aligned}
\label{eq:rho-two-chord-factors}
\end{equation}
where ($a$) follows from equation~\eqref{xueq49} and ($b$) follows from equation~\eqref{xueq48}.  For $0<s<1$, the alternating Taylor series gives $\sin s/s>1-s^2/6$.  Since $R^2\le1/18$, it follows that
\[
g\left(\frac{2R}{\sqrt3}\right)<\frac{81}{80},
\qquad
g(3R)<\frac{12}{11}.
\]
Consequently,
\[
\rho<\frac{81}{80}\cdot\frac{12}{11}\cdot\frac{R}{\sqrt3}
=\frac{243}{220\sqrt3}R
<\frac{2R}{3}.
\]
Using the triangle inequality, we obtain that for any $\bm y\in C$,
\begin{equation*}
\dist(\bm c,\bm y)\leq \dist(\bm c,\bm x)+\dist(\bm x,\bm y)<\frac{2R}{3}+R=	\frac{5R}{3},
\end{equation*}
which gives equation~\eqref{xueq50}.

It remains to prove equation~\eqref{xueq48}. After an additional rotation about the third coordinate axis, we may assume that
\[
\bm x=(\sin\rho,0,\cos\rho).
\]
For every $\bm y\in C$, the triangle inequality gives
\begin{equation}\label{xueq61}
\dist(\bm c,\bm y)
\le
\rho+\dist(\bm x,\bm y)
<
\left(1+\frac{2}{\sqrt3}\right)R
<3R
\le\frac{1}{\sqrt2}
<\frac{\pi}{2},
\end{equation} 
where we used $R^2\le1/18$.  Thus the Karcher energy is smooth near $\bm c$, and its first variation gives
\begin{equation}\label{xueq51}
\int_{C}\xi_{\bm c}(\bm y)\dd\sigma(\bm y)=\bm0.
\end{equation}
For simplicity, denote
\begin{equation*}
w_{\bm y}:=
\frac{\dist(\bm c,\bm y)}
{\sin\dist(\bm c,\bm y)}\in\R,
\qquad
A:=N\int_{C}w_{\bm y}\,\dd\sigma(\bm y)\in\R.	
\end{equation*}
Since $\bm c=(0,0,1)$, equation~\eqref{eq:tangent-displacement} becomes
\[
\xi_{\bm c}(\bm y)
=w_{\bm y}\cdot  (y_1,y_2,0)\in\R^3,
\quad\forall {\bm y}=(y_1,y_2,y_3)\in C,
\]
where the quotient is understood to be $1$ when $\bm y=\bm c$. Consequently, equation~\eqref{xueq51} implies
\begin{equation}
N\int_{C}
w_{\bm y}\cdot 
(y_1,y_2)\dd\sigma(\bm y)
=(0,0).
\label{eq:horizontal-first-variation}
\end{equation}
Using equation~\eqref{eq:horizontal-first-variation} and $(x_1,x_2)=(\sin\rho,0)$, we obtain
\begin{equation}\label{xueq65}
\begin{aligned}
\left\|
\displaystyle
N\int_{C}
w_{\bm y}\cdot 
\bigl((\sin\rho,0)-(y_1,y_2)\bigr)
\dd\sigma(\bm y)
\right\|_{\ell_2}
&=\left\|
\displaystyle
N\int_{C}
w_{\bm y}\cdot 
(\sin\rho,0)
\dd\sigma(\bm y)
\right\|_{\ell_2}\\
&=\left\|
\displaystyle
A\cdot 
(\sin\rho,0)
\right \|_{\ell_2}=A\cdot \sin\rho.	
\end{aligned}
\end{equation}
The function $s/\sin s$ is increasing on $[0,\pi/2)$, with its continuous value $1$ at the origin.  By \eqref{xueq61} we have $\dist(\bm c,\bm y)<3R$, so
\begin{equation}\label{xueq62}
1
\le
w_{\bm y}
<
\frac{3R}{\sin(3R)}.
\end{equation}
Then we have
\begin{equation*}
A=N\int_{C}w_{\bm y}\,\dd\sigma(\bm y)\geq N\int_{C}\,\dd\sigma(\bm y)=1.
\end{equation*}
Moreover,
\begin{equation}\label{xueq63}
\begin{aligned}
\left\|(\sin\rho,0)-(y_1,y_2)\right\|_{\ell_2}
&\le \|\bm x-\bm y\|_{\ell_2} =2\sin\frac{\dist(\bm x,\bm y)}{2} \le \dist(\bm x,\bm y).
\end{aligned}
\end{equation}
Note that
\begin{equation*}
\begin{aligned}
\left\|
\displaystyle
N\int_{C}
w_{\bm y}\cdot 
\bigl((\sin\rho,0)-(y_1,y_2)\bigr)
\dd\sigma(\bm y)
\right\|_{\ell_2}
&\overset{(a)}\leq 
\displaystyle
N\int_{C}
w_{\bm y}\cdot 
\left\|(\sin\rho,0)-(y_1,y_2)\right\|_{\ell_2}
\dd\sigma(\bm y)\\
&\overset{(b)}\leq \frac{3R}{\sin(3R)}N\int_{C}\dist(\bm x,\bm y)\dd\sigma(\bm y)\\
&\overset{(c)}\leq \frac{3R}{\sin(3R)}
\left(
N\int_{C}\dist(\bm x,\bm y)^2\dd\sigma(\bm y)
\right)^{1/2}\\
&\overset{(d)} <
\frac{3R}{\sin(3R)}\frac{R}{\sqrt3}.	
\end{aligned}	
\end{equation*}
Here, ($a$) follows from $w_{\bm y}\geq 1$ and the triangle inequality  for vector-valued integrals, ($b$) follows from \eqref{xueq62} and \eqref{xueq63}, ($c$) follows from Cauchy-Schwarz inequality, and ($d$) follows from the assumption in \eqref{xueq64}. Combining with $A\geq 1$ and \eqref{xueq65}, we obtain equation~\eqref{xueq48}.  This completes the proof.

\end{proof}

\begin{lemma}
\label{lem:weighted-integral-variance}
Let $a>0$ and let $I=[-a,a]$.  For an absolutely continuous function $f:I\to\R$, set
\[
f_*=\frac1{2a}\int_I f(s)\dd s.
\]
Then
\begin{equation}\label{xueq34}
\int_I|f(s)-f_*|^2\dd s
\le
\int_I\frac{a^2-s^2}{2}|f'(s)|^2\dd s.
\end{equation} 
Consequently, if $\Phi:I^2\to\R^n$ is Lipschitz and piecewise $C^1$, and
\[
\Phi_*:=\frac1{(2a)^2}\int_I\int_I\Phi(x,y)\dd x\dd y\in \R^n,
\]
then
\begin{equation}\label{xueq66}
\begin{aligned}
&\int_I\int_I\|\Phi(x,y)-\Phi_*\|_{\ell_2}^2\dd x\dd y\\
&\le
\int_I\int_I
\left(\frac{a^2-x^2}{2}\|\partial_x\Phi(x,y)\|_{\ell_2}^2+
\frac{a^2-y^2}{2}\|\partial_y\Phi(x,y)\|_{\ell_2}^2\right)\dd x\dd y.
\end{aligned}
\end{equation}
\end{lemma}

\begin{proof}
We first prove the one-dimensional inequality. Since $f_*$ is the average of $f$ on $I$, the elementary variance identity gives
\begin{equation*}
\int_I|f(s)-f_*|^2\dd s
=\frac1{4a}\int_I\int_I|f(s)-f(t)|^2\dd s\dd t.
\end{equation*}
Since the integrand is symmetric in $s$ and $t$,  we have
\begin{equation}\label{xueq32}
\int_I|f(s)-f_*|^2\dd s
=\frac2{4a}\int_{-a}^a\int_{s}^{a}|f(s)-f(t)|^2\dd t\dd s.
\end{equation}
For $s,t\in I$ with $s\leq t$, Cauchy-Schwarz on the interval between $s$ and $t$ gives
\begin{equation}\label{xueq31}
|f(s)-f(t)|^2=\left|\int_s^t f'(u)\dd u\right|^2
\le \left(\int_s^t 1\dd u\right)\cdot \left(\int_s^t |f'(u)|^2\dd u\right)= (t-s)\int_{s}^{t}|f'(u)|^2\dd u.
\end{equation}
Substituting equation~\eqref{xueq31} into equation~\eqref{xueq32} gives
\begin{equation*}
\begin{aligned}
\int_I|f(s)-f_*|^2\dd s
&\leq \frac{2}{4a}	\int_{-a}^a\int_{s}^{a}(t-s)\int_{s}^{t}|f'(u)|^2\dd u\dd t\dd s\\
&= \frac{2}{4a}	\int_{-a}^a \left( \int_{-a}^{u}\int_{u}^{a}(t-s)\dd t\dd s\right)\cdot |f'(u)|^2\cdot \dd u\\
&=\int_{-a}^{a}\frac{a^2-u^2}{2}|f'(u)|^2\dd u.
\end{aligned}
\end{equation*}
Hence, we arrive at equation~\eqref{xueq34}.

We next prove \eqref{xueq66}.  For each fixed $y\in I$, define
\[
g(y)
=
\frac1{2a}\int_I\Phi(x,y)\dd x\in\mathbb{R}^n .
\]
For each $(x,y)$, we decompose
\[
\Phi(x,y)-\Phi_*
=
\bigl(\Phi(x,y)-g(y)\bigr)
+
\bigl(g(y)-\Phi_*\bigr).
\]
The two terms are orthogonal after integration in $x$, because
\[
\int_I\bigl(\Phi(x,y)-g(y)\bigr)\dd x=\bm0 .
\]
Therefore
\begin{equation}\label{xueq70}
\begin{aligned}
&\int_I\int_I
\|\Phi(x,y)-\Phi_*\|_{\ell_2}^2
\dd x\dd y  =I_1+I_2.
\end{aligned}
\end{equation}
where
\begin{equation*}
I_1:=\int_I\int_I
\|\Phi(x,y)-g(y)\|_{\ell_2}^2
\dd x\dd y
\quad\text{and}\quad
I_2:=2a\int_I
\|g(y)-\Phi_*\|_{\ell_2}^2\dd y.	
\end{equation*}
We first bound $I_1$. For any fixed $y\in I$, applying equation~\eqref{xueq34}  componentwise gives
\[
\int_I
\|\Phi(x,y)-g(y)\|_{\ell_2}^2\dd x
\le
\int_I
\frac{a^2-x^2}{2}
\|\partial_x\Phi(x,y)\|_{\ell_2}^2\dd x .
\]
Integrating this inequality in $y$ gives
\begin{equation}\label{xueq69}
\begin{aligned}
I_1\le
\int_I\int_I
\frac{a^2-x^2}{2}
\|\partial_x\Phi(x,y)\|_{\ell_2}^2
\dd x\dd y .
\end{aligned}
\end{equation} 
We next bound $I_2$.  The average of $g$ on $I$ is $\Phi_*$, so applying equation~\eqref{xueq34} componentwise to $g$ gives
\begin{equation}\label{xueq67}
2a\int_I\|g(y)-\Phi_*\|_{\ell_2}^2\dd y
\le
2a\int_I
\frac{a^2-y^2}{2}\|g'(y)\|_{\ell_2}^2\dd y .
\end{equation}
Note that
\[
g'(y)
=
\frac1{2a}\int_I\partial_y\Phi(x,y)\dd x
\]
for almost every $y$.  The triangle inequality and Cauchy-Schwarz give
\[
\|g'(y)\|_{\ell_2}^2
=
\left\|
\frac1{2a}\int_I\partial_y\Phi(x,y)\dd x
\right\|_{\ell_2}^2
\le
\frac1{2a}\int_I
\|\partial_y\Phi(x,y)\|_{\ell_2}^2\dd x .
\]
Substituting this bound into equation~\eqref{xueq67}, we obtain
\begin{equation}\label{xueq68}
\begin{aligned}
2a\int_I\|g(y)-\Phi_*\|_{\ell_2}^2\dd y
&\le
\int_I\int_I
\frac{a^2-y^2}{2}
\|\partial_y\Phi(x,y)\|_{\ell_2}^2
\dd x\dd y .
\end{aligned}
\end{equation}
Substituting \eqref{xueq68} and \eqref{xueq69} into \eqref{xueq70} proves \eqref{xueq66}.

\end{proof}

\begin{lemma}
\label{lem:chord-arc-centroid-integral}
Let $C$ be a closed cubed-sphere cell with $\sigma(C)=1/N$ and $C\subset B(\bm x,R)$ for some $\bm x\in\Sph$ and $0<R<\pi/2$.  Define
\[
\bm y_*=N\int_{C}\bm y\dd\sigma(\bm y)
\quad\text{and}\quad
V=N\int_{C}\|\bm y-\bm y_*\|_{\ell_2}^2\dd\sigma(\bm y).
\]
Then $\bm y_*\ne\bm0$.  If $\bm z=\frac{\bm y_*}{\|\bm y_*\|_{\ell_2}}$, then
\[
N\int_{C}\dist(\bm z,\bm y)^2\dd\sigma(\bm y)
\le
\left(\frac R{\sin R\cos\frac{R}{2}}\right)^2
V.
\]
\end{lemma}

\begin{proof}
Since $C\subset B(\bm x,R)$, we see that for any $\bm y\in C$, 
\[
\langle\bm x,\bm y\rangle
=\cos\dist(\bm x,\bm y)\ge\cos R.
\]
Then
\[
\cos R=N\int_{C}\cos R \dd\sigma(\bm y)\le N\int_{C}\langle\bm x,\bm y\rangle\dd\sigma(\bm y)=\langle\bm x,\bm y_*\rangle
\le\|\bm y_*\|_{\ell_2}\le1.
\]
Since $0<R<\pi/2$, we have $\bm y_*\ne\bm0$, so $\bm z$ is well defined.  Moreover,
\[
\langle\bm x,\bm z\rangle
=\frac{\langle\bm x,\bm y_*\rangle}{\|\bm y_*\|_{\ell_2}}
\ge\cos R.
\]
Thus $\dist(\bm z,\bm x)\leq R$.

We next compare geodesic and chordal distances from $\bm z$.  If $\bm y\in C$, then the triangle inequality gives
\[
\dist(\bm z,\bm y)
\le \dist(\bm z,\bm x)+\dist(\bm x,\bm y)\le2R<\pi.
\]
For two points on $\Sph$ at geodesic distance $d$, their Euclidean chord length is $2\sin(d/2)$.  Hence
\[
\|\bm z-\bm y\|_{\ell_2}
=2\sin({\dist(\bm z,\bm y)}/2).
\]
The function $\frac{s}{\sin s}$ is increasing on $0<s<\frac{\pi}{2}$.   Combining with
\begin{equation*}
\dist(\bm z,\bm y)\le \dist(\bm z,\bm x)+\dist(\bm x,\bm y)\leq 2R,	
\end{equation*}
we have
\begin{equation}\label{xueq35}
\dist(\bm z,\bm y)
\le\frac R{\sin R}\cdot \|\bm z-\bm y\|_{\ell_2}.
\end{equation}

It remains to express the Euclidean second moment about $\bm z$ in terms of the variance $V$ about $\bm y_*$.  Using the definition of $\bm y_*$ and $\bm z=\frac{\bm y_*}{\|\bm y_*\|_{\ell_2}}$, a direct expansion gives
\[
\begin{aligned}
V=N\int_{C}\|\bm y-\bm y_*\|_{\ell_2}^2\dd\sigma(\bm y)
=1-\|\bm y_*\|_{\ell_2}^2
\end{aligned}
\]
and
\[
\begin{aligned}
N\int_{C}\|\bm z-\bm y\|_{\ell_2}^2\dd\sigma(\bm y)
&=N\int_{C}
\bigl(\|\bm z\|_{\ell_2}^2-2\langle\bm z,\bm y\rangle
+\|\bm y\|_{\ell_2}^2\bigr)\dd\sigma(\bm y)
=\frac{2V}{1+\|\bm y_*\|_{\ell_2}}.
\end{aligned}
\]
Combining this identity with the chord-arc estimate in equation~\eqref{xueq35} gives
\[
\begin{aligned}
N\int_{C}\dist(\bm z,\bm y)^2\dd\sigma(\bm y)
&\leq \left(\frac R{\sin R}\right)^2 N\int_{C}\|\bm z-\bm y\|_{\ell_2}^2\dd\sigma(\bm y)\\
&=\left(\frac R{\sin R}\right)^2
\frac{2V}{1+\|\bm y_*\|_{\ell_2}}\le\left(\frac R{\sin R\cos\frac{R}{2}}\right)^2
V,
\end{aligned}
\]
because $\|\bm y_*\|_{\ell_2}\ge\cos R$ and $1+\cos R=2\cos^2(R/2)$.

\end{proof}

Now we give a proof of Lemma~\ref{lem:karcher-localization}.

\begin{proof}[Proof of Lemma~\ref{lem:karcher-localization}]
(i)  Recall that $C$ is one equal-area cell.  
Since each face map is obtained from the north-face map $F$ by a rotation, we may assume that $C=F(Q)$ for a closed grid square $Q\subset[-1,1]^2$. Let $\bm p$ be the Euclidean center of $Q$.
Set
\[
\delta=\frac2q,
\qquad
I_\delta=\left[-\frac\delta2,\frac\delta2\right],
\qquad
\bm x=F(\bm p)
\]
and
\[
\Phi(x,y):=F\bigl(\bm p+(x,y)\bigr)\in\Sph,\quad\forall (x,y)\in I_\delta^2.
\]
By equation~\eqref{eq:metric-global-lipschitz}, $\Phi$ is Lipschitz, and its defining formulas make it piecewise $C^1$.  At every point $(x,y)\in I_\delta^2$ at which $\Phi$ is differentiable, define
\begin{equation*}
\widehat{\bm H}(x,y)
=
\bm J_F\bigl(\bm p+(x,y)\bigr)^T
\bm J_F\bigl(\bm p+(x,y)\bigr)
=(\widehat{H}_{j,k})_{j,k=1}^2\in\mathbb{R}^{2\times 2},	
\end{equation*}
where $\bm J_F\in\mathbb{R}^{3\times 2}$ is the Jacobian matrix of the map $F$ defined in equation~\eqref{eq:direct-face-map}. Using the chain rule, we have $\widehat{H}_{1,1}=\|\partial_x\Phi\|_{\ell_2}^2$, $\widehat{H}_{2,2}=\|\partial_y\Phi\|_{\ell_2}^2$, and
\begin{equation*}
\widehat{\bm H}(x,y)=\bm H(\bm p+(x,y)),	
\end{equation*}
where $\bm H$ is the metric matrix used in Lemma~\ref{prop:cubed-sphere-metric}.
%
%
%
By equation~\eqref{eq:metric-direct-formula}, the spherical area Jacobian is $(\det\widehat{\bm H})^{1/2}=\pi/6$ almost everywhere.  The face map is one-to-one away from cell boundaries.  The parameter seams and cell boundaries are planar null sets, and their Lipschitz images are spherical-area null sets. Hence, for every integrable $\Psi$,  we have
\begin{equation}
\begin{aligned}
N\int_{C}\Psi(\bm y)\dd\sigma(\bm y)
&=\frac{N}{4\pi}\int_{I_\delta}\int_{I_\delta}
\Psi(\Phi(x,y))(\det\widehat{\bm H})^{1/2}\dd x\dd y
\\
&=\frac{N}{24}\int_{I_\delta}\int_{I_\delta}
\Psi(\Phi(x,y))\dd x\dd y
=\frac1{\delta^2}\int_{I_\delta}\int_{I_\delta}
\Psi(\Phi(x,y))\dd x\dd y.
\end{aligned}
\label{eq:cell-pullback-integral}
\end{equation}
Here, we use $\dd\sigma=\dd S/(4\pi)$ and $N=6q^2=24/\delta^2$. Take
\[
R:=\frac{3\sqrt3}{8}\delta=\frac{3\sqrt3}{4q}.
\]
For every nonzero $(x,y)\in I_\delta^2$, equation~\eqref{eq:metric-global-lipschitz} gives
\begin{equation}\label{xueq43}
\dist(\bm x,\Phi(x,y))
<\frac{3\sqrt6}{8}\sqrt{x^2+y^2}
\le\frac{3\sqrt6}{8}\cdot \frac{\delta}{\sqrt2}
=R.
\end{equation}
At $(x,y)=(0,0)$ the distance is $0<R$.   Thus,
\begin{equation}
C\subset B(\bm x,R).
\label{eq:cell-support-about-xi}
\end{equation}
Taking $\Psi(\bm y)=\dist(\bm x,\bm y)^2$ in equation~\eqref{eq:cell-pullback-integral} yields
\begin{align*}
N\int_{C}\dist(\bm x,\bm y)^2\dd\sigma(\bm y)
&=\frac1{\delta^2}\int_{I_\delta}\int_{I_\delta}
\dist(\bm x,\Phi(x,y))^2\dd x\dd y\\
&<\frac{27}{32\delta^2}
\int_{I_\delta}\int_{I_\delta}(x^2+y^2)\dd x\dd y=\frac{9}{64}\delta^2=\frac{R^2}{3},
\end{align*}
where the inequality follows from equation~\eqref{xueq43}.  Since $q\ge6$, we have $R^2\le3/64<\frac1{18}$. The containment and moment bounds therefore allow us to apply Lemma~\ref{app-lemma1} to the Karcher center $\bm c$, giving $C \subset B\left(\bm c,\frac{5R}{3}\right)$. Since $\frac{5R}{3}=\frac{15}{\sqrt{8N}}$, assertion \textup{(i)} follows.

(ii) Define
\[
\bm y_*:=N\int_{C}\bm y\dd\sigma(\bm y),
\qquad
V:=N\int_{C}\|\bm y-\bm y_*\|_{\ell_2}^2\dd\sigma(\bm y).
\]
Since $R^2\le1/18$, we have $0<R<\pi/2$.  Applying Lemma~\ref{lem:chord-arc-centroid-integral} with equation~\eqref{eq:cell-support-about-xi} gives $\bm y_*\ne\bm0$.  Thus
\[
\bm z:=\frac{\bm y_*}{\|\bm y_*\|_{\ell_2}}\in\Sph
\]
is well defined.  Then we have
\begin{equation}\label{xueq44}
N\int_{C}\dist(\bm c,\bm y)^2\dd\sigma(\bm y)
\overset{(a)}\le
N\int_{C}\dist(\bm z,\bm y)^2\dd\sigma(\bm y)
\overset{(b)}\le
\left(\frac R{\sin R\cos\frac{R}{2}}\right)^2
V,
\end{equation}
where ($a$) follows from the fact that $\bm c$ minimizes the Karcher energy on $\Sph$, and ($b$) follows from Lemma~\ref{lem:chord-arc-centroid-integral}. We claim that
\begin{equation}
\begin{aligned}
V
\leq\frac{83}{768}\delta^2.
\end{aligned}
\label{eq:cell-euclidean-variance-bound}
\end{equation}
By Taylor estimates we have $\sin R>R-R^3/6$ and $\cos(R/2)>1-R^2/8$.  Combining with $R>0$ and $R^2\le3/64$, we have
\[
\begin{aligned}
\left(\frac{\sin R\cos\frac R2}{R}\right)^2 
&>\left(1-\frac{R^2}{3}\right)
\left(1-\frac{R^2}{4}\right)>1-\frac7{12}R^2
\ge\frac{249}{256}.
\end{aligned}
\]
Combining this estimate with equations~\eqref{eq:cell-euclidean-variance-bound} and~\eqref{xueq44}, we obtain
\begin{align*}
N\int_{C}\dist(\bm c,\bm y)^2\dd\sigma(\bm y)
&<\frac{256}{249}V
\le\frac{256}{249}\cdot \frac{83}{768}\delta^2
=\frac{\delta^2}{9}
=\frac8{3N}.
\end{align*}
This proves equation~\eqref{eq:centered-cell-quarter}.

It remains to prove equation~\eqref{eq:cell-euclidean-variance-bound}. Applying equation~\eqref{eq:cell-pullback-integral}  componentwise gives
\[
\bm y_*=\frac1{\delta^2}\int_{I_\delta}\int_{I_\delta}
\Phi(x,y)\dd x\dd y.
\]
Then, applying Lemma~\ref{lem:weighted-integral-variance} to $\Phi$ with $a=\delta/2$, we have
\begin{equation*}
\begin{aligned}
V&=N\int_{C}\|\bm y-\bm y_*\|_{\ell_2}^2\dd\sigma(\bm y)
 \overset{(a)}=\frac1{\delta^2}\int_{I_\delta}\int_{I_\delta}
\|\Phi(x,y)-\bm y_*\|_{\ell_2}^2\dd x\dd y\\
&\overset{(b)}\le\frac1{\delta^2}\int_{I_\delta}\int_{I_\delta}
\left[
\left(\frac{\delta^2}{8}-\frac{x^2}{2}\right)\widehat{H}_{1,1}
+\left(\frac{\delta^2}{8}-\frac{y^2}{2}\right)\widehat{H}_{2,2}
\right]\dd x\dd y.
\end{aligned}
\end{equation*}
Here, ($a$) follows from equation~\eqref{eq:cell-pullback-integral} by taking $\Psi(\bm y)=\|\bm y-\bm y_*\|_{\ell_2}^2$, and ($b$) follows from $\widehat{H}_{1,1}=\|\partial_x\Phi\|_{\ell_2}^2$ and $\widehat{H}_{2,2}=\|\partial_y\Phi\|_{\ell_2}^2$ almost everywhere. The two weights in the integrand are nonnegative.  If $x^2\ge y^2$, the expression in square brackets equals
\[
\left(\frac{\delta^2}{8}-\frac{x^2}{2}\right)
(\widehat{H}_{1,1}+\widehat{H}_{2,2})+\frac{x^2-y^2}{2}\widehat{H}_{2,2}.
\]
If $y^2\ge x^2$, the analogous identity has the indices interchanged. Since $\widehat{H}_{1,1}+\widehat{H}_{2,2}=\tr\widehat{\bm H}$ and $\widehat{H}_{j,j}\le\lambda_{\max}(\widehat{\bm H})$ for $j=1,2$, it follows that
\[
\begin{aligned}
V\le\frac1{\delta^2}\int_{I_\delta}\int_{I_\delta}
\biggl[&\left(\frac{\delta^2}{8}-\frac{\max(x^2,y^2)}2\right)
\tr\widehat{\bm H}+\frac{|x^2-y^2|}{2}\lambda_{\max}(\widehat{\bm H})\biggr]\dd x\dd y.
\end{aligned}
\]
A direct calculation gives
\[
\begin{aligned}
\int_{I_\delta}\int_{I_\delta}
\left(\frac{\delta^2}{8}-\frac{\max(x^2,y^2)}2\right)\dd x\dd y
=\frac{\delta^4}{16}
\quad\text{and}\quad
\int_{I_\delta}\int_{I_\delta}
\frac{|x^2-y^2|}{2}\dd x\dd y
=\frac{\delta^4}{24}.
\end{aligned}
\]
Moreover, Lemma~\ref{prop:cubed-sphere-metric} implies that almost everywhere,
\[
\tr\widehat{\bm H}\le\frac76,
\qquad
\lambda_{\max}(\widehat{\bm H})\le\frac{27}{32}.
\]
Thus, we have
\begin{equation*}
V \le\delta^2\left(\frac{1}{16}\cdot \frac{7}{6}+\frac{1}{24}\cdot \frac{27}{32}\right)
=\frac{83}{768}\delta^2.	
\end{equation*}
We arrive at equation~\eqref{eq:cell-euclidean-variance-bound}.  This completes the proof.

\end{proof}

\begingroup \raggedright

\endgroup

\end{document}